\documentclass[a4paper]{amsart}

\usepackage{amsmath,amsfonts,amssymb,bm}

\usepackage{cancel}
\usepackage{tikz}
\usetikzlibrary{calc}
\usetikzlibrary{positioning}
\usetikzlibrary{shapes.geometric}
\usepackage{float}
\usepackage[all]{xy}
\def\lcf{\lbrack\! \lbrack}
\def\rcf{\rbrack\! \rbrack}
\usepackage[latin1]{inputenc}
\usepackage{pstricks,pst-node}

\usepackage[numbers,sort&compress]{natbib}
\usepackage{graphicx}

\usepackage{hyperref}
\hypersetup{
   colorlinks   =  true,
    linkcolor    = cyan,
    citecolor    = red,
     urlcolor	=magenta,     
}

\usepackage{amsthm}

\newtheorem{definition}{Definition}[section]

\newtheorem{theorem}[definition]{Theorem}
\newtheorem{proposition}[definition]{Proposition}

\newtheorem{remark}[definition]{Remark}
\newtheorem{example}[definition]{Example}

           {\nolinebreak \hfill $\Box$ \end{trivlist}}

\def \R{\mathbb{R}}

\title[Kirillov structures and reduction of Hamiltonian systems by scaling and standard symmetries]{Kirillov structures and reduction of Hamiltonian systems by scaling and standard symmetries}
\author{A. Bravetti $^1$, S. Grillo$^2$, J. C. Marrero$^3$, E. Padr\'on$^3$}
\thanks{AMS Mathematics Subject Classification (2020): Primary 70G45, 70G65, 70H33; Secundary 53D05, 53D10, 53D17.}
\thanks{Keywords: Hamiltonian systems, contact structures, Kirillov structures, scaling symmetry, standard symmetry, reduction process, reconstruction process. }

\begin{document}

\maketitle

\vspace{-20pt}
\begin{center}
{\small\it  $\;^1$ Instituto de Investigaciones en Matem\'aticas Aplicadas y en Sistemas, Universidad Nacional Aut\'onoma de M\'exico, A.~P.~70543, M\'exico, DF 04510, M\'exico}

{\small\it e-mail:   alessandro.bravetti@iimas.unam.mx}

{\small\it  $^2$ Instituto Balseiro, Universidad Nacional de Cuyo and CONICET
Av. Bustillo 9500, San Carlos de Bariloche
R8402AGP, Rep\'ublica Argentina}

{\small\it e-mail: sergiog@cab.cnea.gov.ar}

{\small\it $\;^3$ULL-CSIC Geometr\'{\i}a Diferencial y Mec\'anica Geom\'etrica, Departamento de Matem\'aticas, Estad\'{\i}stica e Investigaci\'on Operativa
and Instituto de Matem\'aticas y Aplicaciones (IMAULL)}\\{\small\it University of La Laguna, Spain}
\\[5pt]
{\small\it e-mail: jcmarrer@ull.edu.es, mepadron@ull.edu.es}
\end{center}

\begin{abstract} 
In this paper, we discuss the reduction of symplectic Hamiltonian systems by scaling and standard symmetries which commute. We prove that such a reduction process produces a so-called Kirillov Hamiltonian system. Moreover, we show that if we reduce first by the scaling symmetries and then by the standard ones or in the opposite order, we obtain equivalent Kirillov Hamiltonian systems. In the particular case when the configuration space of the symplectic Hamiltonian system is a Lie group $G$, which coincides with the symmetry  group, the reduced structure is an interesting Kirillov version of the Lie-Poisson structure on the dual space of the Lie algebra of $G$. We also discuss a reconstruction process for symplectic Hamiltonian systems which admit a scaling symmetry. All the previous results are illustrated in detail with some interesting examples.

\end{abstract}

\tableofcontents

\section{Introduction}\label{sec:intro}

\subsection{Physical motivation}
The analysis of symmetries is one of the most important tools in theoretical physics. 
Usually, the formulation of a physical theory is given in terms of a variational principle and its associated symplectic Hamiltonian description.
In this context, one typically looks for ``standard symmetries'', that is, symmetries which preserve the symplectic form and the Hamiltonian function.
Among other things, this approach leads to Noether's theorem and its generalization and the Marsden-Weinstein theory of reduction of the system by the action of a symmetry
group (see the classical books and monographs by Marsden and collaborators ~\cite{AM,MMR}, Libermann and Marle  ~\cite{LiMa} or Olver \cite{Olver}).

Recently, there has been a growing interest in the physical literature in considering ``non-standard symmetries'', that is, symmetries of the physical system 
that do not necessarily preserve the symplectic structure. This is motivated mainly by the so-called scaling symmetries and by a well-known philosophical argument 
according to which any minimal description of the universe should avoid introducing a global scale into the picture, that is, 
it should be scale-invariant~\cite{poincare,GrybSloan2021}.
In this context, the theory of ``shape dynamics'' aims to rephrase our best description of the universe (general relativity) in a completely scale-invariant 
fashion~\cite{Barbour,Mercati}.
This has led already to remarkable results that defy the way we understand the (classical) dynamics of the universe. For instance, 
the scale-reduced cosmological and black hole systems can be continued in some cases through the corresponding singularities~\cite{KoslowskiSloan, sloan2019, MercatiSloan}.
Moreover, it has been further argued that the apparent dissipative nature  of the scale-reduced systems may have important consequences for topics
such as the origin of the arrow of time and the formulation of quantum mechanics through unitary operators~\cite{BKM,GrybSloan2021,sloan2021}.

Interestingly, the reduction of a symplectic Hamiltonian system by a scaling symmetry produces a contact Hamiltonian system,
which have been the subject of intensive study recently for their use in the description of e.g.~dissipative, thermostatted and thermodynamic 
systems (see e.g.~\cite{BravettiMechanics,BravettiThermostatted,BravettiThermodynamics,deLeonSardon,deLeonCHS,Gaset1,Gaset2,simoes,Ghosh,Esen1,van} and
the references therein).
This intuition was first put forward in~\cite{sloan2018} and then formalized more precisely in the recent work~\cite{BJS},
where a thorough mathematical investigation 
of the role of scaling symmetries
in symplectic Hamiltonian systems has been performed. 
Moreover, the relationship with the geometry of the blow-ups used in celestial mechanics has also been highlighted, together with the connection with
other geometric structures~\cite{Braddell,Miranda}.

However, so far the study of the joint reduction by scaling and standard symmetries has not been considered in depth, at least from the mathematical perspective.
Moreover, the case in which the reduced manifold is non-orientable, which seems to be the important case 
for the resolution of singularities in general relativity~\cite{KoslowskiSloan, sloan2019, MercatiSloan}, 
has been elusive of a fully-fledged mathematical description (although, see~\cite{BGG,GG,Ma}). 
Finally, from the point of view of comparing the resulting physical theories, 
it is also crucial to highlight how to reconstruct the ``original'' symplectic system from the reduced one.

In this work we perform a detailed mathematical analysis of all the above points.
To give a feeling of the objects involved in our constructions, 
in the remainder of this introduction we provide a high-level description of the most important tools and results.

\subsection{Standard Lie symmetries for Kirillov Hamiltonian systems}
A Kirillov structure on a real line bundle is a Lie algebra structure $[\cdot,\cdot]$ on the space of sections of the dual bundle such that, if we fix a section $h$ on this line bundle, the operator $[\cdot,h]$ is a derivation. Thus, every section of the dual line bundle defines a vector field on the base manifold which is called the Hamiltonian vector field associated with the section. So,  
\emph{a Kirillov Hamiltonian system} is a Kirillov structure on a real line bundle plus a section of the dual bundle (the Hamiltonian section). 

Examples of Kirillov structures may be produced from symplectic, Poisson and Jacobi structures, contact $1$-forms and contact structures (that is, distributions of corank $1$ which are maximally non-integrable). Apart from the last case, in the other previous examples the real line bundle is trivial and the sections of the dual bundle are just $C^\infty$ functions on the base manifold. Anyway, as we show in this paper, there exist interesting examples of Kirillov structures for which the real line bundle is not trivial. In particular, those in which the base space of the line bundle is the projective bundle associated with a vector bundle (for more details on Kirillov structures, see for instance, \cite{G, GMPU, GL, K, Ma}). 

On the other hand, it is well-known that dynamical systems (in particular, mechanical systems), which are invariant under the action of  a symmetry Lie group, have received a lot of attention from researchers in ma\-the\-ma\-tics and physics. For this reason, in this paper we introduce the notion of \emph{a standard Lie symmetry for a Kirillov Hamiltonian system}. It is a principal representation of a Lie group on the line bundle such that the dual representation preserves the Kirillov structure and  the Hamiltonian section is equivariant. A Lie group of symplectic (resp.~Poisson, contact or Jacobi) Hamiltonian symmetries is a particular example of a standard Lie symmetry for the corresponding Kirillov Hamiltonian system. Moreover, for a standard Lie symmetry on a Kirillov Hamiltonian system, the space of orbits of the action on the line bundle is again a line bundle. In fact, in the particular case when the Kirillov structure is Poisson (or Jacobi), we have a reduced Poisson (or Jacobi) structure. This is well-known in the theory of Poisson (or Jacobi) reduction (see, for instance, \cite{MR,Nunes}). 

\subsection{Scaling symmetries for Poisson Hamiltonian systems}

In \cite{BJS} the authors introduce the notion of a scaling symmetry for a symplectic Hamiltonian system and they exhibit several examples where such a symmetry is present
(see also~\cite{sloan2018,BravettiChung,Azuaje2023}).

The previous notion may be extended for the more general class of Poisson Hamiltonian systems as follows. 
It is a principal action $\Phi:\R^{\times}\times P\to P$ of the Lie group $\R^\times$ (with $\R^\times=\R^+$ or $\R^\times =\R-\{0\}$) on the Poisson manifold $(P,\Pi)$ such that 
$$\wedge^2T\Phi\circ \Pi=s \Pi\circ \Phi,\qquad H\circ \Phi=sH,$$ 
for all $s\in \R^\times,$ where $H:P\to \R$ is the Hamiltonian function. In the particular case when $P$ is a symplectic manifold $S$, 
it is proved in \cite{BGG,GG} that the space of orbits $C=S/\R^{\times}$ admits a contact structure. 
In addition, the homogeneous function $H$ on $S$ induces a section of the dual bundle over $C$  
to the Kirillov line bundle in such a way that we have a reduced contact Hamiltonian system (see \cite{GG,Ma}). 

\subsection{Our motivation}
As we mentioned before, many symplectic Hamiltonian systems admit scaling symmetries. 
However, they do not only admit such symmetries, typically they also have standard Lie symmetries. 
In addition, the scaling and the standard Lie symmetries usually commute. 
So, one may reduce the dynamics by both types of symmetries, and some natural questions arise: 
\begin{itemize}
\item What is the nature of the reduced system?
\item If we reduce first by the scaling symmetries and then by the standard ones, is it the same as doing it the other way around?
\item Is it possible to obtain the dynamics of the original symplectic Hamiltonian system from the dynamics of the reduced system via a suitable reconstruction process? 
\end{itemize}
In this paper, we will provide answers to these questions. 

\subsection{The results of the paper}
For a symplectic Hamiltonian system with compatible scaling and standard Lie symmetries (that is, they commute), we will develop two reduction processes: 
\begin{itemize}

\item In the first reduction process, we start with the standard symmetry and then we apply the scaling symmetry. In this case,  the first reduced system is a Poisson Hamiltonian system endowed with a scaling symmetry. The reduction of such a system by this scaling symmetry produces a Kirillov Hamiltonian system (see Theorem  \ref{2reduccion}).

\item In the second reduction process, we use the scaling symmetry and then the standard symmetry.  In this case, the first reduced system is a contact Hamiltonian system endowed with a standard Lie symmetry. The reduction of the latter  by this standard symmetry produces again a final Kirillov Hamiltonian system (see Theorem \ref{p2}). In fact, the reduction of a general Kirillov Hamiltonian system by a standard  Lie symmetry  is again a Kirillov Hamiltonian system (see Theorem \ref{Kirillov}). 

\item We also prove that the final reduced Kirillov Hamiltonian systems obtained in both processes are Kirillov equivalent (see Theorem \ref{equiv}).

\vspace{10pt}

The following diagram summarizes both reduction processes

$$
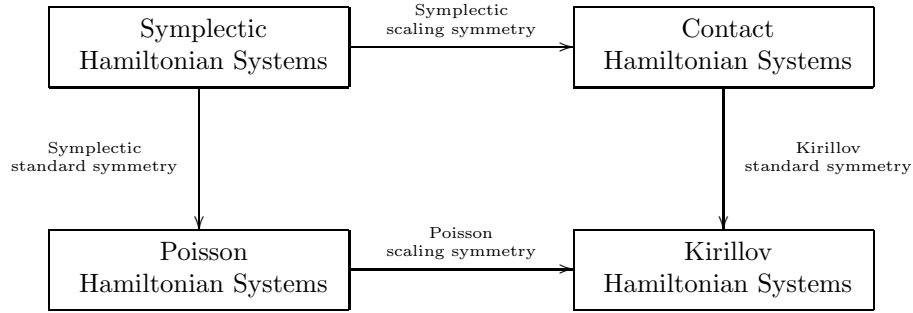
\xymatrix{
*+[F]{\mbox{ \begin{tabular}{c}Symplectic \\ Hamiltonian Systems\end{tabular}}}
\ar[dd]_{\mbox{\tiny{\begin{tabular}{c}Symplectic \\ standard symmetry\end{tabular}}}}
\ar[rrr]^{\mbox{\tiny{\begin{tabular}{c}Symplectic \\ scaling symmetry\end{tabular}}}}
&&& *+[F]{\mbox{ \begin{tabular}{c}Contact \\ Hamiltonian Systems\end{tabular}}} \ar[dd]^{\mbox{\tiny{\begin{tabular}{c}Kirillov \\ standard symmetry\end{tabular}}}}  \\\\
*+[F]{\mbox{ \begin{tabular}{c}Poisson \\ Hamiltonian Systems\end{tabular}}} \ar[rrr]^{\mbox{\tiny{\begin{tabular}{c}Poisson \\ scaling symmetry\end{tabular}}}}
 &&& *+[F]{\mbox{ \begin{tabular}{c}Kirillov \\ Hamiltonian Systems\end{tabular}}}}$$

\item Using more general ideas on reconstruction processes for dynamical systems in the presence of a symmetry Lie group, we present the reconstruction of the symplectic 
(resp.~Poisson) dynamics, for a system which admits a scaling symmetry,  from the reduced contact and (resp.~Kirillov) Hamiltonian dynamics (see Section~\ref{sec:reconstruction}). 

\item All the previous constructions are applied to two examples of symplectic Hamiltonian systems which are  interesting from the physical and mathematical point of view: The 2d harmonic oscillator and standard fiberwise-linear Hamiltonian systems on cotangent bundles induced by vector fields in the configuration space.  For  this last class of examples,  when the cotangent bundle is that of a Lie group $G$, after the two reduction processes,  we obtain an interesting Kirillov structure on the projective space associated with  the dual space ${\mathfrak g}^*$ of the Lie algebra of $G.$ This Kirillov structure may be considered as the Kirillov version of the Lie-Poisson structure on ${\mathfrak g}^*.$ For this reason, it will be called the {\it Lie-Kirillov structure} (see the last part of Subsection \ref{SectionExamplesred}). The geometric nature of this structure and its applications to Hamiltonian dynamics will be discussed in a next paper in progress. We remark that a holomorphic version of the Lie-Kirillov structure has been discussed in \cite{ViWa} (see Examples 54 in \cite{ViWa}). 

\end{itemize}

\subsection{Structure of the paper}
The paper is structured as follows. In Section \ref{sec:intro}, we review some notions and properties  of contact, Poisson, Jacobi and Kirillov manifolds. At the end of the section, a diagram  illustrates the relations between these kinds of structures. 
In Section \ref{ssp}, we show the scaling reduction process of a symplectic (Poisson) Hamiltonian system. This procedure is applied to two examples: The 2d harmonic oscillator and the standard fiberwise-linear Hamiltonian systems on cotangent bundles. 
In Section \ref{section3}, we will discuss the reduction of symplectic Hamiltonian systems which are invariant under the action of a Lie group and, in addition, admit a scaling symmetry which is compatible with the standard symmetry. The reduction process starts by using first the standard symmetry and then the scaling symmetry. The  process in the other direction (the first reduction is obtained by a scaling symmetry and the second one is done using the standard symmetry) is given in Section \ref{Section4}. Moreover, in this section we present a reduction process for general Kirillov Hamiltonian systems in the presence of a standard symmetry. In Sections \ref{section3} and \ref{Section4} both processes are illustrated with the examples mentioned above. The equivalence between the reductions in both directions is proved in Section \ref{sec:equiv}. 
Finally, in Section \ref{sec:reconstruction} we study the reconstruction process by focusing our attention on the case of symplectic Hamiltonian systems with scaling symmetries.

 \section{Contact and Kirillov Hamiltonian systems}
\label{sec:intro}

In this section we recall some  notions and properties of  contact, Jacobi and Kirillov manifolds (for more details see, for instance, \cite{Ar, BGG, DLM, G, GMPU, GL, K, LiMa, Lic, Ma}).

 {\it A contact $1$-form} on a $(2n+1)$-dimensional manifold $C$ is a $1$-form $\eta$ such that $\eta\wedge (d\eta)^n$ defines a volume $1$-form on $C.$  
 We remark that a manifold with a contact $1$-form is orientable and has a distinguished vector field ${\mathcal R}\in {\mathfrak X}(C)$, {\it the Reeb vector field},  
 characterized by the conditions
$$ i_ {\mathcal R}d\eta=0\qquad \mbox{ and } \qquad i_{\mathcal R}\eta=1.$$

The Reeb dynamics can be seen as the one induced by a Hamiltonian vector field on $C$. In fact,  if $H:C\to \R$ is a smooth function on $C$,  
{\it the Hamiltonian vector field $X_H^\eta\in {\mathfrak X}(C)$  of $H$} is characterized by these two conditions 
\begin{equation}\label{v-h}i_{X_H^\eta}d\eta=dH-{\mathcal R}(H)\eta\qquad \mbox{ and }\qquad \eta(X_H^\eta)=H.\end{equation}

The Reeb vector field is just the Hamiltonian vector field for the constant function $H=1.$ 

In the following example we show a  manifold endowed with a contact $1$-form obtained by a reduction process. 

\begin{example}[{\bf The spherical cotangent bundle of a Riemannian manifold}]\label{ex:STQ}\rm
Let $(Q,g)$ be an $n$-dimensional Riemannian manifold and $0_Q$ the zero section of the cotangent  bundle $\tau_Q^*:T^*Q\to Q$. 
On the open subset $T^*Q-0_Q$ of $T^*Q$, we consider the action of the multiplicative group  of the positive real numbers $\R^+$ given by 
\begin{equation}\label{action2}
\phi: \R^+\times (T^*Q-0_Q)\to  (T^*Q-0_Q), \;\;\; \phi(s,\alpha)=s\alpha,
\end{equation}
which defines a principal bundle ${\bf p}:(T^*Q-0_Q)\to (T^*Q-0_Q)/\R^+.$ 
The canonical symplectic structure 
$\omega_Q$ on $T^*Q-0_Q$ is homogeneous with respect to this action, i.e.
 \begin{equation}\label{proy2}
 \phi_s^*(\omega_Q)=s\omega_Q,\,\;\;\; \mbox{ for all } s\in \R^+,
 \end{equation}
or equivalently, 
 $${\mathcal L}_{\Delta_Q}\omega_Q=\omega_Q,$$ where $\Delta_Q$ is the infinitesimal generator of the action $\phi$, that is, $\Delta_Q$ is the Liouville vector field on $T^*Q$. 
 
 The  quotient manifold  $(T^*Q-\{0_Q\})/\R^+$ is diffeomorphic to the spherical cotangent bundle
 $${\Bbb S}(T^*Q)=\{\alpha\in T^*Q/\|\alpha\|=\sqrt{g(\alpha,\alpha)}=1\},$$
 where $g$ denotes here the corresponding metric on $T^*Q.$

 In the particular case when $Q$ is $\R^{n+1},$ with the flat Riemannian metric, we have that the spherical cotangent bundle is 
 	 \begin{equation}\label{eq:ST*Rn1}
	 {\Bbb S}(T ^*\R^{n+1})\cong \R^{n+1}\times S^n\,,
	 \end{equation}
 with $S^n$ the $n$-sphere in $\R^{n+1}.$ 
 
 If $\lambda_Q$ is the Liouville $1$-form on $T^*Q$, i.e. 
 $$\lambda_Q(\alpha)(v)=\alpha({T_\alpha}\tau^*_Q(v)), \mbox{ for all } \alpha\in T^*Q, \; v\in {T_\alpha}(T^*Q),$$
 and $i:{\Bbb S}(T^*Q)\to T^*Q$ is the inclusion map, then $\eta_Q=-i^*\lambda_Q$ is a contact $1$-form on ${\Bbb S}(T^*Q)$ (see, for instance, \cite{B,R,S}). 
 
We remark that the regular and singular Marsden-Weinstein reduction of the spherical cotangent bundle have been discussed some years ago~\cite{DrOrRa,DrRaRo}. 
In fact, this reduction process is a particular case of the more general Marsden-Weinstein contact reduction which has been intensively discussed by several authors \cite{Al,GrGr,GrOr,DrOr,Wi}.
\hfill$\blacklozenge$
\end{example}

A contact $1$-form is a particular case of a Jacobi structure. {\it A Jacobi manifold} $M$ (\cite{K,Lic}) is endowed with a pair $(\Pi,E)\in {\mathcal V}^2(M)\times {\mathfrak X}(M),$ where $\Pi$  is  a $2$-vector field and $E$ is  a vector field on $M$ such that
$$\lcf \Pi,\Pi\rcf=2E\wedge \Pi,\;\;\;\; \lcf E,\Pi\rcf=0,$$
$\lcf\cdot,\cdot\rcf$ being the Schouten-Nijenhuis bracket on $M$. Associated with a Jacobi manifold $(M,(\Pi,E))$ we have a {\it Jacobi bracket}, given by 
\begin{equation}\label{jacobi}
\{f_1,f_2\}_M=\Pi(df_1,df_2) + f_1 E(f_2) -f_2 E(f_1), \mbox{ for } f_1,f_2\in C^\infty(M), \end{equation}
 which is a Lie bracket on the space of functions on $M$ such that  $$\{ff_1,f_2\}_M=f\{f_1,f_2\}_M+f_1\{f,f_2\}_M - f_1f\{1,f_2\}_M,$$
 for $f,f_1,f_2\in C^\infty(M).$ In fact, a Jacobi bracket on the space of functions $C^\infty(M)$ defines a Jacobi structure $(\Pi,E)$ satisfying (\ref{jacobi}). 
  
 Note that we have a vector field  $X_{f_2}^{\{\cdot,\cdot\}_M}$ on $M$, \emph{the Hamiltonian vector field  associated with $f_2$,} such that 
 \begin{equation}\label{JK}
 \{ff_1,f_2\}_M=f\{f_1,f_2\}_M+ X_{f_2}^{\{\cdot,\cdot\}_M}(f) f_1.
 \end{equation} 
  In terms of the Jacobi estructure, this vector field is given by 
 	\begin{equation}\label{eq:XHJacobi}
	X_{f_2}^{\{\cdot,\cdot\}_M}=\Pi(\cdot,df_2)-f_2E\,.
	\end{equation}

 If $E=0$ we recover the notion of a Poisson bracket on the space of functions on $M$  and $(M,\Pi)$ is a Poisson manifold. 
 
 For a manifold $C$ with a contact $1$-form $\eta$,  the Jacobi structure is 
$$
\Pi _\eta (\alpha ,\beta )=d\eta (\flat _\eta ^{-1} (\alpha ), \flat _\eta ^{-1} (\beta )), \qquad \;\;\; E_\eta =-\mathcal{R}$$
for all $\alpha ,\beta \in \Omega^1(C)$, where $\mathcal{R}$ is the Reeb vector field associated with $\eta$ and 
 $\flat _\eta \colon {\mathfrak X}(C)\to \Omega^1(C)$ is the isomorphism of $C^\infty(C)$-modules given by 
$$\flat _\eta(X)=i_Xd\eta +\langle \eta,X\rangle \eta,\;\;\; \mbox{ with }X\in {\mathfrak X}(C).$$ Moreover,  the Hamiltonian vector field defined in (\ref{v-h})
is just the corresponding Hamiltonian vector field $X_{f}^{\{\cdot,\cdot\}_M}$ associated with the Jacobi structure ${(\Pi_\eta,E_\eta})$ (see \cite{Lic}). 

\begin{example}[{\bf continuing Example \ref{ex:STQ}}]{\rm 
 In the case of the spherical cotangent bundle of a Riemannian manifold $(Q,g)$, we consider the differentiable function $\kappa _g:T^*Q-0_Q\to \R$  defined by 
$$\kappa _g(\alpha)=\frac{1}{2} \|\alpha \|^2,\mbox{ for $\alpha\in T^*Q.$}$$

If  $X^{\omega_Q}_{\kappa_{g}}\in {\mathfrak X}(T^*Q-0_Q)$ is the Hamiltonian vector field with respect to  $\omega_Q$ of the function $\kappa_g$, that is, the vector field characterized by 
$$i_{X^{\omega_Q}_{\kappa_{g}}}\omega_Q=d\kappa_g,$$  then 
the Jacobi structure $(\Pi_{\eta_Q}, E_{\eta_Q})$ on $({\Bbb S}(T^*Q),\eta_Q)$   is just the restriction to ${\Bbb S}(T^*Q)$ of the Jacobi structure $(\Pi,E)$  on $T^*Q$ given by  
$$\Pi=\Pi_{\omega_Q} - \Delta_{Q}\wedge X^{\omega_Q}_{\kappa_{g}},\qquad 
E=X^{\omega_Q}_{\kappa_{g}},$$
 where   $\Pi_{\omega_Q}$  is the Poisson structure induced by the symplectic structure $\omega_Q$ on $T^*Q.$ }
 \hfill$\blacklozenge$
\end{example} 

\medskip

On the other hand,  contact $1$-forms are also a particular kind of more general structures which are not, in general, Jacobi structures.

{\it A contact structure} on a $(2n+1)$-dimensional smooth manifold $C$ is a distribution ${\mathcal D}$ on $C$ of codimension~$1$ 
which is maximally non-integrable, i.e. for all $x\in C,$ there is an open neighborhood $U$ of $x$ such that the distribution ${\mathcal D}$ on $U$ 
is given by the annihilator $<\eta_U>^o$  of the vector subbundle of $T^*C$ generated by a contact $1$-form $\eta_U$  on $U$, 
that is  $${\mathcal D}_U=\left\langle \eta_U\right\rangle^o=\{X\in TU/\eta_U(X)=0\}.$$
 In this case, the pair $(C,{\mathcal D})$ is {\it a contact manifold}. 

It is clear that if $C$ has a global contact $1$-form,  the pair $(C,{\mathcal D}=<\eta>^o)$ defines a contact manifold. But in general, a contact structure on $C$ may be not defined by a global contact $1$-form on $C$ as the following example proves.

 \begin{example}[{\bf The projective cotangent bundle of a manifold}]\label{ex:PTQ}\rm
Let $Q$ be an $n$-dimensional manifold and $0_Q$ the zero section of the cotangent  bundle $\tau_Q^*: T^*Q\to Q$. 
On the open subset $T^*Q-0_Q$ of $T^*Q,$ we consider the action of the multiplicative group $\R-\{0\}$ given by 
\begin{equation}\label{action1}
\phi:(\R-\{0\})\times (T^*Q-0_Q)\to  (T^*Q-0_Q), \qquad \phi(s,\alpha)=s\alpha.
\end{equation}

Its infinitesimal generator $\Delta_Q$ is the Liouville vector field on $T^*Q-0_Q$ 
 and the reduced space $(T^*Q-{0_Q})/(\R-\{0\})$ is just the projective cotangent  bundle ${\Bbb P}(T^*Q)$ of $Q.$
 
 \begin{remark}\label{projPV}
The notion of projective bundle ${\Bbb P}(V)$ may be defined for an arbitrary  vector bundle $\tau:V\to Q$ as the quotient bundle induced by 
the action on $V-0_Q$ 
$$(\R-\{0\})\times (V-0_Q)\to (V-0_Q),\qquad (s,v)\to sv, $$
 where $0_Q$ is the zero section of $\tau:V\to Q.$

A  particular case is when $Q$ is a point and $V$ is  the dual of a Lie algebra ${\mathfrak g}.$ In this case, 
the base space of the projective bundle $p:{\mathfrak g}^*-\{0\}\to {\Bbb P}{\mathfrak g}^*$ is just the projective space ${\Bbb P}{\mathfrak g}^*.$  \end{remark}

 If $\lambda_Q$ is the Liouville $1$-form on $T^*Q$ and ${\bf p}:(T^*Q-0_Q)\to {\Bbb P}(T^*Q)$ is the quotient projection, using (\ref{proy2}), 
 one can prove that the distribution of co-rank $1$ 
  $$\widetilde{\mathcal D}=\left\langle \lambda_Q\right\rangle^o$$ 
 is ${\bf p}$-projectable. If  ${\mathcal D}$ denotes  its projection, then $({\Bbb P}(T^*Q),\mathcal{D})$ is a contact manifold. 
 
 A simple example of this kind of contact manifolds is when $Q$ is a Lie group $G$. In this case, the cotangent bundle $T^*G$ 
 may be left trivialized to the trivial vector bundle $G\times {\mathfrak g}^*\to G$, where ${\mathfrak g}$ is the Lie algebra of $G$. 
 Under this identification,  the action $\phi$ is just 
 $$\phi: (\R-\{0\})\times (G\times ({\mathfrak g}^*-\{0\}))\to G\times ({\mathfrak g}^*-\{0\}),\qquad (s,(g,\mu))\to (g,s\mu).$$
 Then, the quotient bundle is ${\bf p}=Id_G\times p:G\times ({\mathfrak g}^*-\{0\})\to G\times {\Bbb P}{\mathfrak g}^*$ and 
 the contact structure is the distribution on $G\times {\Bbb P}{\mathfrak g}^*$ given by 
$${\mathcal D}_{(g,p(\mu))}=\left\langle(T_gL_{g^{-1}})^*(\mu)\right\rangle^o\times T_{p(\mu)}({\Bbb P}{\mathfrak g}^*)$$
for all $g\in G$ and $\mu\in {\mathfrak g}^*-\{0\}.$ Here $L:G\times G\to G$ denotes the left action of the Lie group $G$ on itself. 
 
 \medskip
 
In the particular case when  $G=\R^{n+1}$, the projective cotangent bundle ${\Bbb P}(T^*\R^{n+1})$ can be identified with 
the cartesian product $\R^{n+1}\times {\Bbb P}^n(\R)$, where ${\Bbb P}^n(\R)$ is the real projective space of dimension $n.$ 
 This space is non-orientable when $n$ is even and therefore, ${\Bbb P}(T^*\R^{n+1})$  does not admit a {global} contact $1$-form. 
 \hfill$\blacklozenge$
\end{example}

Contact and Jacobi structures are special examples of more general structures: {\it  Kirillov structures} (see \cite{K}, and also \cite{BGG,G,GG}).

\begin{definition}\rm {\it A Kirillov structure }on a manifold $K$ is a real line bundle $\pi_L:L\to K$ endowed with a Lie bracket $[\cdot,\cdot]_{L^*}:\Gamma(L^*)\times \Gamma(L^*) \to \Gamma(L^*)$ on the space $\Gamma(L^*)$ of sections of the dual line bundle  $\pi_{L^*}: L^*\to K$  such that
$[\cdot,h_2]_{L^*}:\Gamma(L^*)\to \Gamma(L^*)$ is a derivation for all  $h_2\in \Gamma(L^*),$ that is,
\begin{equation}\label{fh}
[fh_1,h_2]_{L^*}=f[h_1,h_2]_{L^*} + X^{[\cdot,\cdot]_{L^*}}_{h_2}(f)h_1, \mbox{ for all }h_1\in \Gamma(L^*)\mbox{ and }f\in C^\infty(K),
\end{equation}
 with $X^{[\cdot,\cdot]_{L^*}}_{h_2}$ a vector field on $K$. The vector field $X^{[\cdot,\cdot]_{L^*}}_{h_2}\in {\mathfrak X}(K)$  is called {\it the symbol of $[\cdot, h_2]_{L^*}$.}
\end{definition}

 The line bundle $(L^*,\pi_{L^*},K)$ with the bracket $[\cdot,\cdot]_{L^*}$  on the space of sections of $\pi_{L^*}$ is,  in Marle's terminology \cite{Ma},  
 a  Jacobi bundle. This kind of structures are  essentially equivalent to the conformal Jacobi structures studied in  \cite{DLM}.
 
When the  line bundle $\pi_L:L\to K$ is trivial, i.e. $L\cong K\times \R$, the sections of $\pi_{L^*}$ can be identified with smooth functions on $K$. 
Under this identification, the local Lie algebra $[\cdot,\cdot]_{L^*}$ is a Lie bracket 
$$\{\cdot,\cdot\}_K:C^\infty(K)\times C^\infty(K)\to C^\infty(K)$$
satisfying that, for all $f\in C^\infty(K),$ 
$$\{f f_1,f_2\}_K=f\{f_1,f_2\}_K + X^{\{\cdot,\cdot\}_K}_{f_2}(f)f_1,$$
for all $f_1,f_2\in C^\infty(K).$ 

Note that if $f_1=1$ then $\{f,f_2\}_K=f\{1,f_2\}_K + X^{\{\cdot,\cdot\}_K}_{f_2}(f)$, which implies 
$$\{f f_1,f_2\}_K=f\{f_1,f_2\}_K + f_1\{f,f_2\}_K-ff_1\{1,f_2\}_K.$$ This means that $\{\cdot,\cdot\}_K$ is a Jacobi bracket, 
whose associated Jacobi structure $(\Pi,E)$  is given by 
$$E(f_1)=\{1,f_1\}_K \qquad \mbox{ and }\qquad  \Pi(df_1,df_2)=\{f_1,f_2\}_K -f_1\{1,f_2\}_K + f_2\{1, f_1\}_K,$$
with $f_1,f_2\in C^\infty(K).$ Conversely, every Jacobi manifold $(K,\{\cdot,\cdot\}_K)$  
defines a Kirillov structure on the trivial line bundle $\pi:K\times \R\to K.$ Therefore, Jacobi structures are just  trivial Kirillov structures.  

In the case of a contact manifold $(C,{\mathcal D}),$ consider the line bundle with total space the annihilator bundle ${\mathcal D}^o$ of ${\mathcal D},$ 
$\pi_{ {\mathcal D}^o}: {\mathcal D}^o\to C$  of ${\mathcal D}$, which is, in general, not trivial.  Using this line bundle and the representation $\R^\times\times \R\to \R$  of $\R^\times$ (with $\R^\times=\R^+$ or $\R^\times=\R-\{0\}$)  over the vectorial space of real numbers  given by $$(s,t)\to \frac{t}{s},$$ we have a $\R^\times$-principal bundle ${\bf p}:S:= ({\mathcal D}^o-0_C)\to C\cong S/\R^\times$ (see Appendix \ref{A}). Here, $0_C$ is the zero section of $\pi_{ {\mathcal D}^o}: {\mathcal D}^o\to C.$ Moreover, we may consider the $1$-form $\lambda_S$ on $S$
$$\lambda_S(\alpha)(v)=<\alpha,T_\alpha {\bf p}(v)>,\;\;\; \mbox{ with } \alpha\in ({\mathcal D}^o-0_C), \,\,\,v\in T_\alpha ({\mathcal D}^o-0_C),$$
which defines the symplectic structure $\omega_S=-d\lambda_S.$ This symplectic structure is homogeneous with respect 
to the  $\R^\times$-action  $\phi^S:\R^\times \times S\to S$ on $S,$ i.e.
$$(\phi_s^S)^*(\omega_S)=s\omega_S, \mbox{ for } s\in \R^\times.$$

Now, a Lie bracket $[\cdot,\cdot]_{({\mathcal D}^o)^*}$ on the space of sections $\Gamma( ({\mathcal D}^o)^*)$ of the line bundle $({\mathcal D}^o)^*\to C$ can be constructed as follows.

There is a one-to-one correspondence between the sections of $\pi_{ ({\mathcal D}^o)^*}: ({\mathcal D}^o)^*\to C$ and the homogeneous functions $H:S\to \R$ on $S$ satisfying
$$H\circ \phi^S_s=sH, \,\;\; \mbox{ for } s\in \R^\times,$$
(see Appendix \ref{A}). Using the homogeneous character of the symplectic structure $\omega_S,$ we deduce that the Poisson bracket $\{H_1,H_2\}_S$ induced by $\omega_S$ of two homogeneous functions $H_1,H_2:S\to \R$ is again a homogeneous function. Taking into account this fact, we define the Kirillov bracket 
 $[\cdot,\cdot]_{({\mathcal D}^o)^*}:\Gamma(({\mathcal D}^o)^*)\times \Gamma(({\mathcal D}^o)^*)\to \Gamma(({\mathcal D}^o)^*)$
by the relation 
\begin{equation}\label{KS}
\{H_1,H_2\}_S=-H_{[h_{H_1},h_{H_2}]_{({\mathcal D}^o)^*}},
\end{equation}
where $h_{H_i}$ is the section of $\pi_{({\mathcal D}^o)^*}:({\mathcal D}^o)^*\to C$ associated with the homogeneous function $H_i$ on $S$ and $H_{[h_{H_1},h_{H_2}]_{({\mathcal D}^o)^*}}$ is the homogenous function associated with the section $[h_{H_1},h_{H_2}]_{({\mathcal D}^o)^*}.$
In conclusion, every contact manifold $(C,{\mathcal D})$ admits a Kirillov structure on the line bundle $\pi_{{\mathcal D}^o}: {\mathcal D}^o\to C$. 

The following diagram illustrates the relations among all the previous geometric structures.

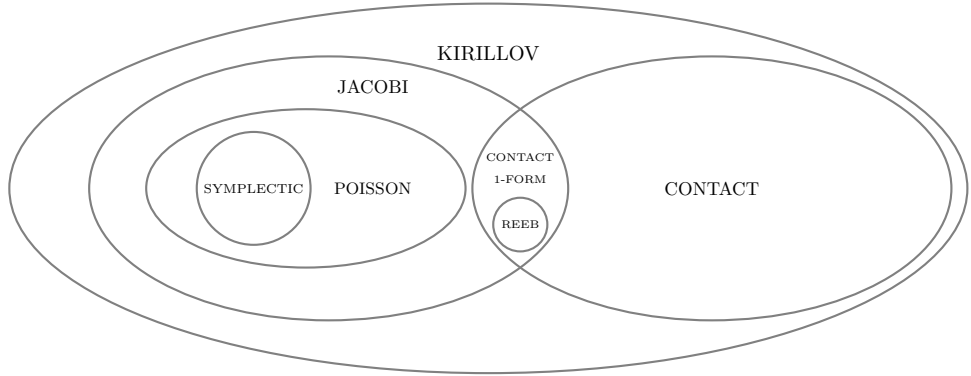
\begin{center}
\tikzstyle{elps}=[ellipse,draw=black!50,thick]
\tikzstyle{crcl}=[circle,draw=black!50,thick]
\begin{tikzpicture}[thick,scale=0.6,every node/.style={scale=0.7}]

\node[elps,minimum width=18cm,minimum height=7cm] (G) at (3.5,0) {};

\node[elps,minimum width=9cm,minimum height=5cm] (G) at (0,0) {};
\node[elps,minimum width=6cm,minimum height=3cm] (G) at (-0.5,0) {};
\node[elps,minimum width=9cm, minimum height =5cm] (H) [right=2pt of G] {CONTACT};
\node[crcl,minimum width=2cm] (ker) [left=60pt of H] {\scriptsize{SYMPLECTIC}};
\node[] (Poisson) [left=20pt of H] {\small{POISSON}};
\node[above=1cm of Poisson] (Jacobi) {{JACOBI}};
\node[crcl,minimum width=0.5cm,label=90:\begin{tabular}{c}{\tiny{CONTACT}}\\ {\tiny{$1$-FORM}}\end{tabular}] (p) at (4.2,-0.8)
{\tiny{REEB}};
\node[] (Kirillov)  at (3.5,3)
 {\large{KIRILLOV}};
\end{tikzpicture}
\end{center}

 
\section{Scaling symmetries and symplectic (Poisson) Hamiltonian systems}\label{ssp}

In the previous examples, the reduction processes  are the fundamental tool  to obtain contact structures from symplectic structures. Now, we will show this process for a general symplectic Hamiltonian system, which was discussed in  \cite{GG}, and then we will present some examples. 
We begin by recalling the notion of scaling symmetries \cite{BJS} for this kind of dynamical systems. 

\begin{definition}
Let $(S,\omega)$ be a symplectic manifold and $H:S\to \R$ a function on $S$. A scaling symmetry for the dynamical system $(S,\omega,H)$ 
is a principal action $\phi:\R^\times\times S\to S$ of the multiplicative group $\R^\times$ (with $\R^\times=\R^+$  or $\R^\times=\R-\{0\})$ on $S$ such that 
$$\phi_s^*\omega=s\omega\qquad \textrm{and}\qquad \phi_s^*H=sH,\qquad \textrm{for all}\,\,\, s\in \R^\times.$$ 
\end{definition}
Note that if $\Delta\in {\mathfrak X}(S)$ is the infinitesimal generator of the scaling symmetry, then 
$${\mathcal L}_{\Delta}\omega=\omega\qquad \mbox{ and }\qquad {\mathcal L}_{\Delta} H=H.$$ 

In fact, if $\R^\times$ is connected (that is, $\R^\times =\R^+$),
 then the previous conditions are equivalent to the fact that the principal action $\phi$ is a scaling symmetry. 

An immediate consequence of the existence of a scaling symmetry is that the symplectic structure is exact, that is, 
$\omega=-d\lambda$ with $\lambda=-i_\Delta\omega.$
Moreover, the $1$-form $\lambda$ is homogeneous, i.e. $(\phi_s)^*\lambda =s\lambda,$
and if $\Pi_\omega$ is the Poisson bi-vector induced by $\omega,$ then $\Pi_\omega$ satisfies the following relation 
\begin{equation}\label{Pw}
\wedge^2T\phi_s\circ \Pi_\omega={s}\Pi_\omega\circ\phi_s,
\end{equation}
where $\wedge^2T\phi_s:\wedge^2TS\to \wedge^2TS$ is the vector bundle isomorphism induced by the diffeomorphism $\phi_s:S\to S.$ 

Now, we will develop the reduction process with the scaling symmetry $\phi.$ 

Denote by  $C:= S/{\Bbb R}^\times$ the corresponding quotient   manifold and by ${\bf p_S}:S\to C$ its  quotient projection. Then, we may consider the distribution 
$$\widetilde{\mathcal D}=\left\langle\lambda\right\rangle^o,$$ which is ${\bf p}$-projectable and the corresponding distribution ${\mathcal D}$ on $C$, which is a contact structure. 

Denote by $[\cdot,\cdot]_{({\mathcal D}^o)^*}$  the Kirillov bracket on the space of sections of the line bundle $\pi_{({\mathcal D}^o)^*}:({\mathcal D}^o)^*\to C$ characterized by (\ref{KS}). On the other hand, from the homogeneity of $H:S\to\R$ with respect to the scaling symmetry, we have a section $h_H:C\to ({\mathcal D}^o)^*$ of $\pi_{({\mathcal D}^o)^*}$. The corresponding symbol $X_{h_H}^{[\cdot,\cdot]_{({\mathcal D}^o)^*}}$ of $h_H$ given as in (\ref{fh}) is just the ${\bf p}$-projection on $C$ of the Hamilton vector field $X_H^{\omega}.$  The following diagram summarizes this reduction process (see \cite{GG}, for more details on this reduction process).
\[\xymatrix{ 
&&\R&\\
&
(S,\omega,X^{\omega}_{H})\kern-20pt\ar@(ul,u)[]^{\R^\times}\kern40pt\ar[d]^{{\bf p_S}} \ar[ur]^{H}
&&\\&
 (C,{\mathcal D},X_{h_H}^{[\cdot,\cdot]_{({\mathcal D}^o)^*}}) \ar[rr]_{\kern30pt h_H}&&({\mathcal D}^o)^*\ar@<-1ex>[ll]_{\kern30pt \pi_{({\mathcal D}^o)^*}}}
\]

Now, we will exhibit two examples of contact dynamical  systems  induced by a scaling reduction process. 
 
 \begin{example}[\bf{The 2d harmonic oscillator and the spherical cotangent bundle}]\label{ex:STQho}\rm  

 Consider the manifold $Q=\R^2-\{(0,0)\}$, which is diffeomorphic  to $\R^+\times S^1$ via the map 
 \begin{equation}\label{id}
 \Psi:\R^2-\{(0,0)\}\to \R^+\times S^1,\qquad \Psi(q)=(\|q\|, \frac{q}{\|q\|}).
 \end{equation}
 Then, under this identification the space $T^*Q-0_Q\cong (\R^2-\{(0,0)\})\times(\R^2-\{(0,0)\})$ is just $(\R^+\times S^1)\times (\R^+\times S^1)$. Moreover,  if $(r,\theta)$  (respectively, $(r,\theta, r', \theta')$) are polar coordinates on $Q\cong \R^+\times S^1$ (respectively, on $T^*Q-0_Q\cong \R^+\times S^1\times \R^+\times S^1$), we have that the local expression of 
 the standard symplectic form $\omega_Q$ and the corresponding Poisson bi-vector $\Pi_{\omega_Q}$   on $\R^+\times S^1\times \R^+\times S^1$ are  respectively
 $$\omega_Q=cos(\theta-\theta')dr\wedge dr'+ r'\sin(\theta-\theta')dr\wedge d\theta'-r\sin(\theta-\theta')d\theta\wedge dr' + rr'cos(\theta-\theta')d\theta\wedge d\theta'$$
 and
 \begin{equation}\label{wP}
 \begin{array}{lcr}
 \Pi_{\omega_Q}&=& -\cos(\theta-\theta'){\partial_r}\wedge {\partial_{r'}}- \displaystyle\frac{\sin(\theta-\theta')}{r'}{\partial_r}\wedge{\partial_{\theta'}}+\displaystyle\frac{\sin(\theta-\theta')}{r}{\partial_\theta}\wedge {\partial_{r'}} - \displaystyle\frac{\cos(\theta-\theta')}{rr'}{\partial_\theta}\wedge {\partial_{\theta'}}.
 \end{array}
 \end{equation}
 
Now,  we consider the symplectic Hamiltonian system $(T^*Q,\omega_Q,H)$ 
 of the harmonic oscillator where, under the identification (\ref{id}),
 $H:T^*Q\to\R $ is the Hamiltonian function  given by 
	\begin{equation}\label{eq:Hho}
	H(r,\theta,r',\theta')=\frac{1}{2}\left(r^2+(r')^{2}\right)\,, 
	\end{equation} 
with $r,r'\in \R^+.$  In this case the dynamics is given by the  Hamiltonian vector field 
$$X_{H}^{\omega_{Q}}=r\cos(\theta-\theta') {\partial_{r'}}+ r\frac{\sin(\theta-\theta')}{r'}{\partial_{\theta'}}-r'{\cos(\theta-\theta')}{\partial_r} + r'\frac{\sin(\theta-\theta')}{r} {\partial_\theta}.$$

We consider the action of $\R^+$ on $\R^+\times S^1$, whose infinitesimal generator is 
$$\Delta=\frac{1}{2} (r{\partial_r} + r'{\partial_{r'}}).$$
Note that it defines a scaling symmetry, since ${\mathcal L}_{\Delta}\omega_Q=\omega_Q$ and ${\mathcal L}_{\Delta}H=H.$

On the other hand, the diffeomorphism
$$\begin{array}{rcl}
\R^+\times S^1\times \R^+\times S^1&\to& \R^+\times S^1\times \R^+\times S^1\\
(r,\theta,r',\theta')&\to & (\rho,\theta,\rho',\theta')=(r,\theta, \displaystyle\frac{r'}{r},\theta')
\end{array}
$$
transforms the generator $\Delta$ of the ${\Bbb R}^+$-action on $\R^+\times S^1\times \R^+\times S^1$ 
into the vector field $\frac{1}{2} \rho\partial_\rho$. 
The inverse of this map is $(\rho,\theta,\rho',\theta')\to (\rho, \theta,\rho\rho',\theta')$. 
Then,  we have that:
\begin{itemize}
\item The reduced space  ${\Bbb S}(T^*(\R^+\times S^1))$ (see Example \ref{ex:STQ})
is diffeomorphic to $\R^+\times S^1\times S^1.$ Under this identification, the quotient map ${\bf p}:\R^+\times S^1\times \R^+\times S^1\to \R^+\times S^1\times S^1$ is just  $${\bf p}(\rho,\theta,\rho',\theta')=(\rho',\theta,\theta')$$

\item The contact $1$-form under this identification  is given by
$$\eta=\iota^*(i_{\Delta}\omega_Q)=\frac{1}{2}\big(\rho'\sin(\theta-\theta')(d\theta + d\theta')+\cos(\theta-\theta')d\rho'\big)$$
with $(\rho',\theta,\theta')\in \R^+\times S^1\times S^1.$ Here $\iota:\R^+\times S^1\times S^1\to \R^+\times S^1\times \R^+\times S^1$ is the inclusion $\iota(\rho',\theta,\theta')=(1,\theta,\rho',\theta').$

The Reeb vector field associated with this contact $1$-form is 
$${\mathcal R}=2\cos(\theta-\theta')\partial_{\rho'}+2\displaystyle\frac{\sin(\theta-\theta')}{\rho'}\partial_{\theta'}.$$

From the homogeneity of the Poisson structure $\{\cdot,\cdot\}_{\omega_Q}$ with respect to the symplectic form $\omega_Q$ we deduce that 
$$\{\rho^2h,\rho^2h'\}_{\omega_Q}=\frac{1}{2}\rho{\partial_\rho}\{\rho^2h,\rho^2h'\}_{\omega_Q},$$ with 
$h,h'\in C^\infty(\R^+\times S^1\times S^1).$ Therefore, 
$$\{\rho^2h,\rho^2h'\}_{\omega_Q}=\rho^2\{h,h'\}_{C},$$
where  $\{\cdot,\cdot\}_C$ is the Jacobi bracket on $C=\R^+\times S^1\times S^1$ and $h,h'\in C^\infty(C).$ 

From this fact and using the local expression of $\Pi_{\omega_Q}$ with respect to the coordinates $(\rho,\theta,\rho',\theta')$, we obtain the Jacobi bracket associated with the contact structure defined by $\eta$

\begin{equation}\label{BJ}
\begin{array}{rcl}
\{h,h'\}_C&=&-2\cos(\theta-\theta')(h\partial_{\rho'}h'-h'\partial_{\rho'}h)-2\displaystyle\frac{\sin(\theta-\theta')}{\rho'}(h\partial_{\theta'} h' -h'\partial_{\theta'} h)\\[8pt]&& \sin(\theta-\theta')(\partial_{\rho'}h\partial_{\theta'}h'-\partial_{\rho'}h'\partial_{\theta'}h) + \sin(\theta-\theta')(\partial_{\theta}h\partial_{\rho'}h'-\partial_{\theta}h'\partial_{\rho'}h)\\[8pt]&& +\displaystyle\frac{\cos(\theta-\theta')}{\rho'}(\partial_{\theta}h\partial_{\theta'}h'-\partial_{\theta}h'\partial_{\theta'}h)
\end{array}
\end{equation}

Therefore, the Jacobi structure  is given by 
\begin{equation}\label{JC}
\begin{array}{rcl}
\Pi_C&=&\sin(\theta-\theta')\partial_{\rho'}\wedge \partial_{\theta'} - \sin(\theta-\theta')\partial_{\rho'}\wedge \partial_{\theta}-\displaystyle\frac{\cos(\theta-\theta')}{\rho'}\partial_{\theta}\wedge \partial_{\theta'},\\[5pt] E_C&=&-2\cos(\theta-\theta')\partial_{\rho'}-2\displaystyle\frac{\sin(\theta-\theta')}{\rho'}\partial_{\theta'}.
\end{array}
\end{equation}

\item The reduced Hamiltonian function $H$ is the function $H_{|\R^+\times S^1\times S^1}(\rho',\theta,\theta')=\displaystyle\frac{1}{2}((\rho')^2+1)$. 
\item The reduced vector field on $\R^+\times S^1\times S^1$ is 
\begin{equation}\label{p1}T{\bf p}(X^{\omega_Q}_H)=(1+(\rho')^2)\cos(\theta-\theta')\partial_{\rho'}+\sin(\theta-\theta')(\frac{1}{\rho'}\partial_{\theta'}+\rho'\partial_{\theta}),\end{equation}
which is just the contact  Hamiltonian vector field of the restriction $H_{|\R^+\times S^1\times S^1}$ with respect to the contact $1$-form $\eta$ or, equivalently, the Jacobi Hamiltonian vector field  of $H_{|\R^+\times S^1\times S^1}$ with respect to $(\Pi_C,E_C).$ 
\end{itemize}
 \hfill$\blacklozenge$
 \end{example}

In the previous example the Hamiltonian function $H$ induces a  function $H_{|C}$ on the reduced space $C$. 
However, in general, we do not necessarily have a function on the reduced space, as the following example proves.

 \begin{example}\label{ex:projective0}{\bf The projective cotangent Hamiltonian system deduced from a standard linear Hamiltonian system. }{\rm Let  $Y\in {\mathfrak X}(Q)$ be a vector field on the manifold $Q$ of dimension $n.$ We denote by $Y^{\ell}:T^*Q\to \R$ the fiberwise-linear function induced by $Y$, i.e.
 \begin{equation}\label{ell}
 Y^\ell(\alpha)=<\alpha,Y(\tau^*_Q(\alpha))>,\;\;\; \forall \alpha\in T^*Q,
 \end{equation} 
 with $\tau_Q^*:T^*Q\to Q$ the canonical projection. 
 If $(q^i,p_i)$ are local coordinates of $T^*Q-0_Q$, the local expression of $Y^\ell$ is 
  $$Y^\ell(q,p)=Y^i(q)p_i,$$
  where  $Y(q)=Y^i(q){\partial_{ q^i}}.$ 
   We remark that  the linearity of $Y^\ell$ implies its homogeneity, i.e.
  $$Y^\ell(s\alpha) =sY^\ell(\alpha),\mbox{  for all $s\in \R-\{0\}$ and $\alpha\in T^*Q$,}$$
   with respect to the action given in (\ref{action1}).    
   
    The local expression of the Hamiltonian vector field $X^{\omega_Q}_{Y^\ell}\in {\mathfrak X}(T^*Q)$ with respect to the canonical symplectic structure $\omega_Q$ on $T^*Q$ is 
   $$X^{\omega_Q}_{Y^\ell}=Y^k\partial_{q^k} - p_j{\partial_{q^k}}Y^j\partial_{p_k}.$$
   
   Moreover, if $\{\cdot, \cdot\}_{\omega_Q}$ is the Poisson bracket induced by $\omega_Q$, then 
   $$\{Y^\ell, Z^\ell\}_{\omega_Q}=-[Y,Z]^\ell,$$ for all  $Y,Z\in\mathfrak{X}(Q).$
   
  Let $U_{i_0}$ be the open subset of $T^*Q-0_Q$ given by
$$U_{i_0}=\{(q^1,\dots, q^n,p_1,\dots, p_{n})\in T^*Q-0_Q/p_{i_0}\not=0\}.$$ 
   
Then, if $H$ is the restriction of $Y^\ell$ to $T^*Q-0_Q,$ after the reduction process of the  symplectic  
Hamiltonian system $(T^*Q-0_Q,\omega_Q,H)$ by the scaling symmetry, 
we have that:
\begin{itemize}
\item  The corresponding reduced space is the  projective cotangent bundle ${\bf p}:T^*Q-0_Q\to {\Bbb P}(T^*Q)$ induced by the action (\ref{action1}). If we denote by $\widetilde{p}=(\widetilde{p}_1,\dots ,\widetilde{p}_{i_0-1},\widetilde{p}_{i_0+1},\cdots,\widetilde{p}_n)$ the standard coordinates on ${\bf p}(U_{i_0})\subseteq {\Bbb P}(T^*Q)$, then the local expression of the projection ${\bf p}$  on $U_{i_0}$ is 
$${\bf p}(q^1,\dots q^n, p_1,\dots p_n)=(q^1,\dots q^n,\frac{p_1}{p_{i_0}}, \dots,\frac{p_{i_0-1}}{p_{i_0}}, \frac{p_{i_0+1}}{p_{i_0}},\dots, \frac{p_n}{p_{i_0}})=(q,\widetilde{p}).$$

\item The contact distribution ${\mathcal D}$  on ${\bf p}(U_{i_0})$ is just

$$
\begin{array}{rcl}
({\mathcal D}_{(q,\widetilde{p})})_{|{{\bf p}(U_{i_0})}}
=&&T_{(q,p)}{\bf p}(<p_idq^i>^o)=T_{(q,p)}{\bf p}<X_1,\dots,X_{i_0-1},X_{i_0+1},\dots, X_{n},
\partial_{p_1}, \dots, \partial_{p_n}>\\[5pt]
=&&<\widetilde{X}_1,\dots,\widetilde{X}_{i_0-1},\widetilde{X}_{i_0+1},\dots, \widetilde{X}_{n},
\partial_{\widetilde{p}_1},\dots, \partial_{ \widetilde{p}_{i_0-1}},\partial_{\widetilde{p}_{i_0+1}},\cdots,\partial_{\widetilde{p}_n}>,
\end{array}
$$
with $X_i=p_{i}\partial_{q^{i_{0}}}-p_{i_{0}}\partial_{q^{i}}$, 
$\widetilde{X}_i=\widetilde{p}_{i}\partial_{q^{i_{0}}}-{\widetilde{p}_{i_{0}}}\partial_{q^{i}}.$ Moreover, the local expression of the line bundle $\pi_{{\mathcal D}^o}:{\mathcal D}^o\to  {\Bbb P}(T^*Q)$  on  ${\bf p}(U_{i_0})$  is 
$$\pi_{{\mathcal D}^o}(q,\widetilde{p},t)=(q,\widetilde{p}).$$

\item The section $h_{Y^\ell}:{\Bbb P}(T^*Q)\to ({\mathcal D}^o)^* $ of $\pi_{({\mathcal D}^o)^*}:({\mathcal D}^o)^*\to  {\Bbb P}(T^*Q)$ associated with $Y^\ell$ is defined locally by 
\begin{equation}\label{ell}
h_{Y^\ell}(q,\widetilde{p})(q,\widetilde{p},t)=Y^\ell(q,\widetilde{p}_1,\cdots, \widetilde{p}_{i_0-1}, t, \widetilde{p}_{i_0+1}\cdots, \widetilde{p}_{n})=Y^i(q)\widetilde{p}_i + Y^{i_0}(q)t.
\end{equation}
\item The Kirillov bracket $[\cdot,\cdot]_{({\mathcal D}^o)^*}$ on the sections of the dual of the line bundle $\pi_{{\mathcal D}^o}$ satisfies the condition
$$[h_{X^\ell}, h_{Y^\ell}]_{({\mathcal D}^o)^*}=-h_{\{X^\ell,Y^\ell\}_{\omega_Q}}=h_{[X,Y]^\ell}.$$

\item The Hamiltonian vector field  $X^{\omega_Q}_{Y^\ell}\in {\mathfrak X}(T^*Q)$ is ${\bf p}$-projectable  to a vector field on ${\Bbb P}(T^*Q)$ whose local expression is 
$$Y^i\partial_{{q}^i}+ \big(\widetilde{p}_j(\widetilde{p}_i\partial_{{q}^{i_0}}Y^j - \partial_{{q}^i}Y^j)+ \widetilde{p}_i\partial_{q^{i_0}}Y^{i_0} -\partial_{{q}^i}Y^{i_0}
\big)\partial_{\widetilde{p}_i}.$$
\end{itemize}
 \hfill$\blacklozenge$
 
 {\bf The particular case of a Lie group.} When $Q$ is a Lie group $G$ and the vector field $Y$ on $G$ is left-invariant,  we have (see Example \ref{ex:PTQ}): 
\begin{itemize}
\item The vector field $Y$ is given by $Y(g)=T_eL_{g}(\xi),$ with $\xi$ an element of the Lie algebra ${\mathfrak g}$ of $G$. 
\item The linear function $Y^\ell:G\times {\mathfrak g}^*\to \R$ is just 
$Y^\ell(g,\alpha)=\alpha(\xi).$
\item The reduced space is $G\times {\Bbb P}{\mathfrak g}^*.$ 
\item The contact structure is the distribution on $G\times {\Bbb P}{\mathfrak g}^*$ 
given by 
$${\mathcal D}_{(g,p(\mu))}=\left\langle(T_gL_{g^{-1}})^*(\mu)\right\rangle^o\times T_{{p}(\mu)}({\Bbb P}{\mathfrak g}^*)\mbox{ for all $g\in G$ and $\mu\in {\mathfrak g}^*-\{0\}$}.$$
 Here $p:{\mathfrak g}^*-\{0\}\to {\Bbb P}{\mathfrak g}^*$ is the corresponding quotient map determined by the scaling symmetry on   ${\mathfrak g}^*-\{0\}.$ 
\item 
The fiber of the  line bundle $\pi_{{\mathcal D}^o}: {\mathcal D}^o\to G\times {\Bbb P}{\mathfrak g}^*$  at $(g,\mu)\in G\times {\Bbb P}{\mathfrak g}^*$ is just 
$$
{{\mathcal D}^o}_{(g,p(\mu))}=\left\langle(T_gL_{g^{-1}})^*(\mu)\right\rangle.$$

\item 
The reduced Hamiltonian section of $\pi_{({\mathcal D}^o)^*}: ({\mathcal D}^o)^*\to G\times {\Bbb P}{\mathfrak g}^*$ induced by $Y^\ell$ is 
$$h_\xi(g,p(\mu))(t(T_gL_{g^{-1}})^*(\mu))=t\mu(\xi)$$
with $g\in G$, $\mu\in {\mathfrak g}^*-\{0\}$ and $\xi=Y(e).$

\item Under the identification $T^*G-0_G\cong G\times ({\mathfrak g}^*-\{0\}),$ the symplectic structure $\omega_G$ is given by 
$$\omega_G(g,\mu)((v_1,\mu_1), (v_2,\mu_2))= -\mu_1(T_gL_{g^{-1}}(v_2))+\mu_2(T_gL_{g^{-1}}(v_1))+\mu[T_gL_{g^{-1}}(v_1), T_gL_{g^{-1}}(v_2)]_{{\mathfrak g}}$$ 
for all $g\in G$, $\mu, \mu_1,\mu_2\in {\mathfrak g}^*$ and $v_1,v_2\in T_gG$ (see \cite{AM}). Here $[\cdot,\cdot]_{\mathfrak g}$
 is the Lie algebra structure on ${\mathfrak g}.$ Then,  the Hamiltonian vector field $X_{Y^\ell}^{\omega_G}\in {\mathfrak X}(T^*G-0_G)$   
can be identified with the pair $$(Y,\{\cdot,\xi^\ell\}_{{\mathfrak g}^*-\{0\}})\in {\mathfrak X}(G)\times {\mathfrak X}({\mathfrak g}^*-\{0\}),$$ 
where ${\xi^\ell}$ is the restriction to ${\mathfrak g}^*-\{0\}$ of the linear function ${\xi}^\ell: {\mathfrak g}^*\to \R$ 
induced by $\xi$ and $\{\cdot,\cdot\}_{{\mathfrak g}^*-\{0\}}$ is the restriction  to functions on ${\mathfrak g}^*-\{0\}$ of the Lie-Poisson bracket on ${\mathfrak g}^*$.  
We recall that this bracket is characterized by 
\begin{equation}\label{Lie-Poisson}
\{{\xi}^\ell_1,{\xi}^\ell_2\}_{{\mathfrak g}^*}(\alpha)=-\alpha([\xi_1,\xi_2]_{{\mathfrak g}}),
\end{equation}
with $\alpha\in {\mathfrak g}^*$ and $\xi_i\in {\mathfrak g}$ (for more details,  see \cite{AM}). 

 The reduced vector field after this reduction is just $(Y, X_{h_\xi})\in {\mathfrak X}(G)\times {\mathfrak X}({\Bbb P}{\mathfrak g}^*),$ such that 
\begin{equation}\label{h1}
X_{h_\xi}(f)\circ p=\{f\circ p, {\xi^\ell}\}_{{\mathfrak g}^*-\{0\}}, \quad\forall f\in C^\infty({\Bbb P}{\mathfrak g}^*),
\end{equation}
which is the symbol of the derivation $[\cdot, h_\xi]_{({\mathcal D}^o)^*}.$

A more explicit (local) expression of the vector field $X_{h_\xi}\in {\mathfrak X}({\Bbb P}{\mathfrak g}^*)$ may be obtained as follows. For each   $\nu\in {\mathfrak g}-\{0\}$ one can consider the coordinate open  neighborhood 
$p(U)$ of ${\Bbb P}{\mathfrak g}^*$ with $U=\{\alpha\in {\mathfrak g}^*/\nu^\ell(\alpha)=\alpha(\nu)\not=0\}$. On $p(U)$  the typical local coordinates  in ${\Bbb P}{\mathfrak g}^*$ have the form  ${\emph r}(\zeta,\nu)$ 
 characterized by 
$${\emph r}(\zeta,\nu)\circ p=\frac{\zeta^\ell}{\nu^\ell},$$
with $\zeta\in {\mathfrak g}-\{0\}.$ Moreover, using (\ref{Lie-Poisson}) and (\ref{h1}), we deduce that 
\begin{equation}\label{21'}
X_{h_\xi}({\emph r}(\zeta,\nu))\circ p=\frac{\zeta^\ell ([\nu,\xi]_{\mathfrak g})^\ell - \nu^\ell ([\zeta,\xi]_{\mathfrak g})^\ell}{({\nu^\ell})^2}.
\end{equation}

\end{itemize}

 }
\end{example}

Given the above facts, it is natural to ask if it is possible to extend the previous reduction to a  Poisson Hamiltonian system, not necessarily symplectic. 
The following result gives an 
affirmative answer to this question. Before that, we introduce the notion of scaling symmetry for this kind of systems. 

\begin{definition}\label{def:scalingPoisson} \rm
 If $(P,\Pi,H)$ is a  Poisson Hamiltonian system on the Poisson manifold $(P
 ,\Pi)$,  {\it a scaling symmetry for  
 $(P,\Pi,H)$}  is a principal  action $\phi^P: \R^\times \times P\to P$ of the multiplicative group
  $\R^\times$ (with $\R^\times=\R^+$ or $\R^\times=\R-\{0\}$)  on $P$ such that  the Poisson structure $\Pi$ and the function 
  $H$ are homogeneous with respect to  the action $\phi^P$, that is,
 \begin{equation}\label{Ps}
\wedge^2T\phi^P_s\circ \Pi={s}\Pi\circ \phi_s^P \quad \mbox{ and }\quad H\circ \phi^P_s=sH\quad \mbox{ for } s\in \R^\times,
  \end{equation}
where $\wedge^2T\phi^P_s: \wedge^2 TP\to \wedge^2TP$ is the vector bundle isomorphism induced by $\phi_s^P:P\to P$. 
\end{definition}

The conditions in (\ref{Ps}) are equivalent to the following ones 
\begin{equation}\label{Ps2}
\{F\circ \phi_s^P,G\circ \phi_s^P\}_P=s(\{F,G\}_P\circ \phi_s^P) \quad\mbox{ and } \quad H\circ \phi^P_s=sH\,,\quad \mbox{ for } F,G\in C^\infty(P) \mbox{ and } s\in \R^\times,
\end{equation}
where $\{\cdot,\cdot\}_P$ is the Poisson bracket of functions on $P.$ 

We remark that (\ref{Ps}) implies that the Poisson structure $\Pi$ and the Hamiltonian function $H$  satisfy (see \cite{DLM,Ma})
$${\mathcal L}_{\Delta_P}\Pi=-\Pi\quad \mbox{ and }\quad {\mathcal L}_{\Delta_P}H=H,$$
where $\Delta_P$ is the infinitesimal generator of $\phi^S.$ Moreover, if $\R^\times$ is connected (that is, $\R^\times=\R^+$) the previous conditions are equivalent to (\ref{Ps}). 
In addition, in the case of a symplectic manifold $(S,\omega),$ the condition $$\wedge^2T\phi^P_s\circ \Pi_\omega={s}\Pi_\omega\circ \phi_s^P,$$ 
with $\Pi_\omega$ the Poisson structure induced by $\omega, $ is equivalent to 
$(\phi_s^P)^*\omega=s\omega.$

We have the following important result. 


\begin{theorem}\label{reductionPoisson}\rm
Let ${\bf p_P}: P\to K=P/\R^\times$ be a principal $\R^\times$-bundle with total space a homogeneous Poisson manifold $(P,\Pi)$.  
If $\pi_L:L\to K$ is the line bundle associated with the principal bundle ${\bf p_P}$ (see Appendix \ref{A}), then: 
\begin{enumerate}

\item[$a)$] There is a one-to-one correspondence between homogeneous functions $H:P\to \R$ and sections $h_H:L^*\to K$ of the dual line bundle $\pi_{L^*}:L^*\to K$ of $\pi_L$.

\item[$b)$]  On the space $\Gamma(L^*)$ of the sections of the line bundle $\pi_{L^*}:L^*\to K$,  we have a Kirillov bracket $$[\cdot,\cdot]_{L^*}:\Gamma(L^*)\times \Gamma(L^*)\to \Gamma(L^*)$$
such that the Poisson bracket $\{H_1,H_2\}_P$ of two homogeneous functions $H_1,H_2:P\to \R$ is just  
$$\{H_1,H_2\}_P=-H_{[h_{H_1},h_{H_2}]_{L^*}},$$
where $H_{[h_{H_1},h_{H_2}]_{L^*}}$ is the homogeneous function on $P$ associated with ${[h_{H_1},h_{H_2}]_{L^*}}.$
\item[$c)$]  The Hamiltonian vector field $X_H^{\{\cdot,\cdot\}_P}=-i(dH)\Pi\in {\mathfrak X}(P)$ of a homogeneous function $H$ with respect to the Poisson bracket $\{\cdot,\cdot\}_P$ is ${\bf p_P}$-projectable and its projection is the symbol $X_{h_H}^{[\cdot,\cdot]_{L^*}}\in {\mathfrak X}(K)$ of the derivation $[\cdot,h_H]_{L^*},$ i.e.  the following diagram is commutative 
\[
\xymatrix{
P\ar[d]^{X^{\{\cdot,\cdot\}_P}_H}\ar[rr]^{\bf p_P}&&K\ar[d]^{X_{h_H}^{[\cdot,\cdot]_{L^*}}}\\
TP\ar[rr]^{T{\bf p_P}}&&TK}
\]

\item[$d)$]  We have that 
$$
\big[X^{[\cdot,\cdot]_{L^*}}_{h_1}, X^{[\cdot,\cdot]_{L^*}}_{h_2}\big]=-X^{[\cdot,\cdot]_{L^*}}_{[h_1,h_2]_{L^*}},$$
for all $h_1,h_2\in \Gamma(L^*).$
\end{enumerate}
\end{theorem}

\begin{proof}
 
For a proof of $a)$ see Appendix \ref{A}.

\noindent  If $H_1,H_2$ are two homogeneous functions then,
$$H_1\circ \phi^P_s=sH_1 \mbox{ and } H_2\circ \phi^P_s=sH_2\,,$$
and, using (\ref{Ps2}), we deduce that 
$$\{H_1\circ \phi^P_s, H_2\circ \phi^P_s\}_P=s(\{H_1,H_2\}_P\circ \phi_s^P),$$
which implies that
$$s\{H_1,H_2\}_P=\{H_1,H_2\}_P\circ\phi_s^P,$$ that is, 
the function $\{H_1,H_2\}_P$ is homogeneous. Thus, 
the Poisson bracket $\{\cdot,\cdot\}_P$ is closed for homogeneous functions with respect to $\Delta_P.$

Using this fact and Proposition \ref{1-1}, (see Appendix \ref{A}) we may define a bracket 
$[ \cdot,\cdot]_{L^*}:\Gamma(L^*)\times \Gamma(L^*)\to \Gamma(L^*)$ on the space $\Gamma(L^*)$ of the sections of $\pi_{L^*}:L^*\to K$ which is characterized by 
\begin{equation}\label{charac}
H_{[h_1,h_2]_{L^*}}=-\{H_{h_1},H_{h_2}\}_P, \mbox{ with }h_1,h_2\in \Gamma(L^*).
\end{equation} 

This bracket was described (up to the sign) in \cite{BGG} (Theorem 3.2). 
Using the fact that the Poisson bracket $\{\cdot,\cdot\}_P$ defines a  Lie algebra on the space of functions on $P$ and (\ref{charac}), we deduce that  $[\cdot,\cdot]_{L^*}$ is a Lie bracket. Moreover, for a $C^\infty$ function $f:K\to \R$, from the properties of the Poisson bracket $\{\cdot,\cdot\}_P$, we have that 
\begin{equation}\label{D1}
\begin{array}{rcl}
H_{[fh_1,h_2]_{L^*}}&=&-\{H_{fh_1},H_{h_2}\}_P=-\{(f\circ {\bf p_P})H_{h_1},H_{h_2}\}_P\\[5pt]&=&-(f\circ {\bf p_P})\{H_{h_1},H_{h_2}\}_P- \{(f\circ {\bf p_P}),H_{h_2}\}_P H_{h_1}\\[5pt]&=&(f\circ {\bf p_P})H_{[h_1,h_2]_{L^*}}+ X^{\{\cdot,\cdot\}_P}_{ H_{h_2}}(f\circ {\bf p_P})H_{h_1}.\end{array}
\end{equation}
On the other hand, using the homogeneity of $\Pi$ and $H_{h_2},$ we deduce that
$${\mathcal L}_{\Delta_P} X_{H_{h_2}}^{\{\cdot,\cdot\}_P}=-{\mathcal L}_{\Delta_P} i_{dH_{h_2}}\Pi=-i_{dH_{h_2}}{\mathcal L}_{\Delta_P} \Pi-i(d(\Delta_P(H_{h_2})))\Pi=0,$$
or, equivalently,  $X^{\{\cdot,\cdot\}_P}_{ H_{h_2}}$ is ${\bf p_P}$-projectable. Then, there is a vector field $X_{h_2}^{[\cdot,\cdot]_{L^*}}$ on $K$ such that 
\begin{equation}\label{proyX}X^{[\cdot,\cdot]_{L^*}}_{h_2}\circ {\bf p_P}=T{\bf p_P}\circ X^{\{\cdot,\cdot\}_P}_{ H_{h_2}}. \end{equation}
From (\ref{D1}) and (\ref{proyX}), we have that 
$$H_{[fh_1, h_2]_{L^*}}=(f\circ {\bf p_P})H_{[h_1,h_2]_{L^*}}+ (X^{[\cdot,\cdot]_{L^*}}_{h_2}(f)\circ {\bf p_P})H_{h_1},$$
and consequently (see (\ref{fh}))  we have a Kirillov structure on the space of sections of $\pi_{L^*}:L^*\to K$ and the symbol of $[\cdot,h]_{L^*}$ is just the ${\bf p_P}$-projection on $K$ of the Hamiltonian vector field $X^{\{\cdot,\cdot\}_P}_{H_h}$. 
This proves $b)$ and $c).$

Finally, from (\ref{proyX}) and  using that $\big[X^{\{\cdot,\cdot\}_P}_{H_{h_1}}, X^{\{\cdot,\cdot\}_P}_{H_{h_2}}\big]=-X^{\{\cdot,\cdot\}_P}_{\{H_{h_1},H_{h_2}\}_P},$
we have that 
$$
\big[X^{[\cdot,\cdot]_{L^*}}_{h_1}, X^{[\cdot,\cdot]_{L^*}}_{h_2}\big]=-X^{[\cdot,\cdot]_{L^*}}_{[h_1,h_2]_{L^*}}.$$
Therefore, we deduce $d).$
\end{proof}

\begin{remark} In \cite{Ma} Marle proves that if $\pi_L:L\to K$ is a line bundle endowed with a Kirillov structure 
--~$(L^*,\pi_{L^*},K)$ is a Jacobi bundle in his terminology~-- 
and $h:K\to L^*$  is a section of $\pi_{L^*}$,  
then one can induce a Poisson structure $\Pi$ on $L^*$ (which is homogeneous with respect to the negative of the Euler vector field $\Delta$ on $L^*$),  
a differentiable function $H:P:=(L^*-0_{L^*})\to \R$ and a vector field $X$ on $L^*$ such that: 
\begin{itemize}
\item
The restriction of $X$ to $P$ is just the Hamiltonian vector field induced by $\Pi$ and $H.$ 
\item 
The vector field  $X$  projects on a vector field $X_h$ on  $K$
\end{itemize}
(see Theorem 4.3 and Proposition 4.7 in \cite{Ma}). Therefore, if the flow of $\Delta$ induces a principal action on $P,$ then we have a Poisson Hamiltonian system $(P,\Pi,H)$ with a scaling symmetry in such a way that the corresponding reduced Kirillov Hamiltonian system is just the original system. So, Marle's result may be considered as a converse of Theorem \ref{reductionPoisson}. 
\end{remark}
\begin{remark}
In \cite{ViWa2} (see Theorem 2.2.6 of \cite{ViWa2}), the authors obtain a one-to-one correspondence between Atiyah $(l,m)$-tensors on a line bundle and homogeneous $(l,m)$-tensors on its slit dual bundle (the dual bundle with the zero section removed). Using this general result, one could prove that there exists a one-to-one correspondence between Kirillov structures on the line bundle and homogeneous Poisson structures on its slit dual bundle  (see Example 2.4.2 in \cite{ViWa2}). Anyway, in order to have our paper more self-contained, we have included a direct and simple proof of the items $a), b), c)$ and $d)$ of Theorem \ref{reductionPoisson}.

\end{remark}
The following diagram summarizes Theorem \ref{reductionPoisson}
\[\xymatrix{ 
&&\R&\\
&
(P,\{\cdot,\cdot\}_P,X^{\{\cdot,\cdot\}_P}_{H})\ar@(ul,u)[]^{\R^\times}\ar[d]^{{\bf p_P}} \ar[ur]^{H}
&&\\&
 (K=P/\R^\times,\pi_L:L\to K,X_{h_H}^{[\cdot,\cdot]_{L^*}}) \ar[rr]_{\kern80pt h_H}&&L^*\ar@<-1ex>[ll]_{\kern80pt \pi_{L^*}}}
\]

\section{Reduction of  symplectic Hamiltonian systems using first  standard symmetries and then scaling symmetries}\label{section3}

 In this section, we will discuss the reduction of symplectic Hamiltonian systems which are invariant under the action of a symmetry Lie group and, in addition, admit a scaling symmetry. 
 The standard and the scaling symmetries will be compatible in the following sense. 
\begin{definition}{\rm 
Let $(S,\omega,H)$ be a symplectic Hamiltonian system. Suppose  that  $\phi^S:\R^\times \times S\to S$ is a scaling symmetry on  $(S,\omega,H)$. Additionally, suppose that we have a Lie group $G$  and a $G$-principal bundle $\wp_S:S\to S/G$  such that    the  corresponding  action $\Phi^S:G\times S\to S$ on the symplectic manifold $S$ satisfies:
\begin{enumerate}
\item[(i)] $(\Phi^S_g)^*(\omega)=\omega,$ for $g\in G,$ i.e. the action $\Phi^S$ is symplectic.

\item[(ii)] $H:S\to \R$ is $G$-invariant, that is, $H(\Phi^S(g,x))=H(x),$ for all $x\in S$ and $g\in G.$ 

\item[(iii)] The symplectic and  the scaling actions commute, that is, $\Phi^S_g\circ \phi_s^S=\phi_s^S\circ \Phi^S_g,$  for all $s\in \R^\times$ and $g\in G.$ 
\end{enumerate}
In this case we say that the dynamical system {\it $(S,\omega,H)$  admits  a scaling symmetry $\phi^S:\R^\times \times S\to S$ and a symplectic $G$-symmetry $\Phi^S:G\times S\to S$ which are compatible}. }
\end{definition}
 
 Note that the previous conditions (i) and (ii)  imply that 
\begin{equation}\label{sp}
{\mathcal L}_{\xi_S}\Pi_{\omega}=0 \quad\mbox{ and }\quad {\mathcal L}_{\xi_S}H=0,
\end{equation}
where $\xi_S$ is the infinitesimal generator of the action $\Phi^S$ associated with an element $\xi$ of the Lie algebra ${\mathfrak g}$ of $G$ and $\Pi_\omega$ is the Poisson bi-vector on $S$ induced by the symplectic structure $\omega$. If $G$ is connected, then the conditions (i) and (ii)  are equivalent to (\ref{sp}). 

\subsection{The first step: Reduction by standard symmetries.}
It is well-known (see \cite{MR}) that  the symplectic structure on $S$ induces a Poisson bracket $\{\cdot,\cdot\}_P$ on the quotient manifold $P:=S/G$ characterized by
\begin{equation}\label{pc=cp}
\{f_1\circ \wp_S,f_2\circ \wp_S\}_S=\{f_1,f_2\}_P\circ \wp_S.
\end{equation}
with $f_i\in C^\infty(P)$, where $\{\cdot,\cdot\}_S$ is the Poisson bracket induced by the symplectic structure $\omega$  on $S.$ Consequently, the Poisson structure $\Pi_P$ on $P$ and the Poisson structure $\Pi_\omega$ induced by the symplectic structure $\omega$ are related as follows 
\begin{equation}\label{spPI}
\wedge^2T\wp_S\circ \Pi_\omega=\Pi_P\circ \wp_S
\end{equation}

In addition, from the $G$-invariance of $H$, there is a reduced Hamiltonian function $H^G:P\to \R$ such that
\begin{equation}\label{HG1}
H^G\circ \wp_S=H.
\end{equation}

Moreover, the Hamiltonian vector field  $X_{H}^\omega\in {\mathfrak X}(S)$ is $\wp_S$-projectable 
and its projection is just the Hamiltonian vector field $X_{H^G}^{\{\cdot,\cdot\}_P}=\{\cdot, H^G\}_P\in {\mathfrak X}(P)$ associated with the Poisson structure $\Pi_P.$ 

The following diagram summarizes this first reduction process

\[
\xymatrix{ 
&&\R&\\
&
(S ,\omega,X^{\omega}_{H})\kern-20pt\ar@(ul,u)[]^{G}\kern20pt\ar[d]^{\wp_S} \ar[ur]^{H}
&&\\&
 (P=S/G,\{\cdot,\cdot\}_P, X_{H^G}^{\{\cdot,\cdot\}_P}) \ar[ruu]_{H^G}&&
}
\]
  
On the other hand,  using that $\Phi^S_g\circ \phi_s^S=\phi_s^S\circ \Phi^S_g,$  for all $s\in \R^\times$ and $g\in G, $ the $\R^\times$-action $\phi^S$ induces an  action $\phi^{P}:\R^\times\times P\to P$ characterized by 
\begin{equation}\label{aP}
\phi_s^{P}(\wp_S(x))=\wp_S(\phi^S_s(x)), \mbox{ for all } x\in S\mbox{ and } s\in \R^\times.
\end{equation}

Then, we have
\begin{proposition}
$\phi^P$ is a scaling symmetry for the Poisson Hamiltonian system $(P,\Pi_P, H^G).$ 
\end{proposition}
\begin{proof}
Given $s\in \R^\times$, using (\ref{spPI}) and (\ref{aP}), it follows that
$$\wedge^2T\phi^P_s\circ \Pi_P\circ \wp_S=\wedge^2T\phi^P_s\circ \wedge^2T\wp_S\circ \Pi_\omega=\wedge^2T\wp_S\circ \wedge^2T\phi_s^S\circ \Pi_\omega.$$
Now, since $\phi^S$ is a scaling symmetry for the symplectic manifold $(S,\omega)$, we deduce that 
$$\wedge^2T\phi^P_s\circ\Pi_P\circ \wp_S=s\wedge^2T\wp_S\circ \Pi_\omega \circ \phi_s^S$$
and, using again (\ref{spPI}), we obtain that 
$$\wedge^2T\phi^P_s\circ \Pi_P\circ \wp_S=s\Pi_P\circ \phi_s^P\circ \wp_S.$$
This implies that 
$$\wedge^2T\phi^P_s\circ \Pi_P=s\Pi_P\circ \phi_s^P.$$
On the other hand, from (\ref{HG1}) and (\ref{aP}), it follows that 
$$H^G\circ \phi^P_s\circ \wp_S=H^G\circ \wp_S\circ \phi_s^S=H\circ \phi_s^S$$
and, since $H$ is a homogeneous function for the action $\phi^S,$ we deduce that 
$$H^G\circ \phi_s^P\circ \wp_S=sH=sH^G\circ \wp_S,$$
where for the last equality we use again  (\ref{HG1}). 
This implies that 
$$H^G\circ \phi_s^P=sH=sH^G,$$
which ends the proof of the result. 
\end{proof}

 Now, we may apply the scaling reduction process.

\subsection{The second step: Reduction by scaling symmetry.}
Consider the Poisson Hamiltonian system $(P,\Pi_P,$ $H^G)$ obtained in the previous subsection by reduction from  the  symplectic Hamiltonian system $(S,\omega,H).$ In  the second step of the reduction process we will apply Theorem \ref{reductionPoisson} to the Poisson Hamiltonian system $(P,\Pi_P,H^G)$ and the scaling symmetry $\phi^P:\R^\times \times P\to P.$

The complete reduction process is described in the following theorem.

\begin{theorem}\label{2reduccion}\rm
Let $(S,\omega,H)$ be a symplectic Hamiltonian system with compatible  scaling symmetry  $\phi^S:\R^\times \times S\to S$ and  symplectic $G$-symmetry $\Phi^S:G\times S\to S$, $G$ being a Lie group. Then:

\begin{enumerate}
\item The multiplicative group $\R^\times$ acts on the Poisson manifold $P=S/G$ 
such that  the corresponding quotient map ${\bf p_{P}}:P\to P/\R^\times$ is a $\R^\times$-principal bundle. 
Moreover, if $\pi_L:L\to K=P/\R^\times$ is the line bundle associated with ${\bf p_P}:P\to K=P/\R^\times$, 
then the homogeneous function $H^G:P\to  \R$ induces a section $h_{H^G}:K\to L^*$ of the dual line bundle $\pi_{L^*}:L^*\to K$ of $\pi_L$.
 
 \item On the space of sections $\Gamma( {L}^*)$ of  $\pi_{L^*}:L^*\to  K$,  we have a Kirillov bracket 
 $$[\cdot,\cdot]_{ {L}^*}:\Gamma( {L}^*)\times \Gamma( {L}^*)\to \Gamma( {L}^*)$$ 
 such that, if $\{\cdot,\cdot\}_P$ is the Poisson bracket on $P,$ 
$$
 [h_{H^G_1},h_{H^G_2}]_{ {L}^*}=-h_{\{H_1^G,H_2^G\}_P},
$$
 for  $H_1^G,H_2^G\in C^\infty(P)$ homogeneous functions  on $P.$ 
 \item The Hamiltonian vector field $X_H^\omega$ is $({\bf p_P}\circ \wp_S)$-projectable on $K$ and its projection is the symbol $X^{[\cdot,\cdot]_{ {L}^*}}_{h_{H^G}}\in{\mathfrak X}(K) $ of $[\cdot,h_{H^G}]_{ {L}^*}$.  \end{enumerate}

\end{theorem}

 The following diagram illustrates both reduction processes together.

 
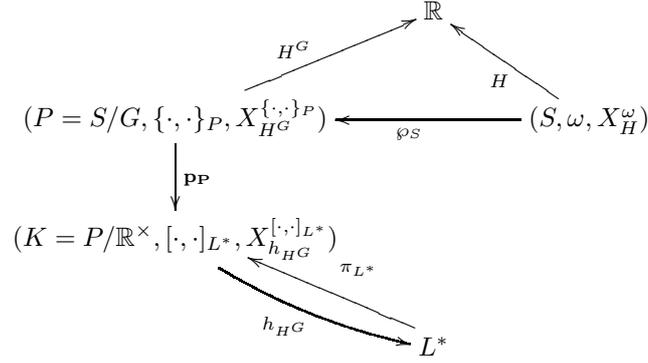
\[
\xymatrix{&\R&&\\
(P=S/G,\{\cdot,\cdot\}_P,X^{\{\cdot,\cdot\}_P}_{H^G})  \ar[d]^{\bf p_{P}}\ar[ur]^{H^G}&&(S,\omega,X_H^\omega)\ar[ul]^H\ar[ll]^{\kern20pt \wp_S}
&\\
(K=P/\R^\times,[\cdot,\cdot]_{  L^*}, X_{h_{H^G}}^{[\cdot,\cdot]_{  L^*}})\ar@/_/[dr]_{h_{H^G}}&&&
\\&{  L}^*\ar@<-1ex>[ul]_{\kern20pt \pi_{  L^*}}&&}
\]

\subsection{Examples}\label{SectionExamplesred}
In this subsection we will apply the previous reduction processes to Examples \ref{ex:STQho} and \ref{ex:projective0}.
\begin{example}\label{ex:2dho1}{\bf Continuing Example \ref{ex:STQho}: The 2d harmonic oscillator reduced first by a standard and then by a scaling symmetry.} {\rm
In this case we have:
\begin{enumerate}
\item A standard rotational $S^1$-symmetry, with infinitesimal generator $\xi_S=x\partial_{y}-y\partial_{x}+p_{x}\partial_{p_{y}}-p_{y}\partial_{p_{x}},$ where $(x,y,p_x,p_y)$ are coordinates on $S=T^*(\R^2-\{(0,0)\})$. Using the identification $\R^+\times S^1\times \R^+ \times S^1\cong T^*(\R^2-\{(0,0)\})-0_{\R^2-\{(0,0)\}},$  the local expression of $\xi_S$ is 
$$\xi_{S}=\partial_\theta +\partial_{\theta'},$$
where $(r,\theta,r',\theta')$ are polar coordinates on $\R^+\times S^1\times \R^+ \times S^1.$ 
 \item A scaling $\R^{+}$-symmetry, with generator $$\Delta_S=\frac{1}{2} (r{\partial_r} + r'{\partial_{r'}}).$$
\end{enumerate}

One can also easily check that $[\xi_S,\Delta_{S}]=0$ and thus, since the multiplicative group $\R^+$ and $S^1$ are connected, the two symmetries commute. 
Therefore, the corresponding actions are \emph{compatible} 
and we can apply Theorem~\ref{2reduccion}. In order to highlight all the mechanisms involved, we will proceed by steps and indicate the main derivations.

In the first step, with the $S^1$-symmetry, the reduced objects are: 
\begin{itemize}
\item{\bf The reduced space:}
We perform the reduction by the standard symmetry, obtaining the Poisson system $(P,\Pi_P,H^{G})$.  
Firstly, we have that the symplectomorphism 
$$\begin{array}{rcl}\R^+\times S^1\times \R^+ \times S^1&\to& S^1\times( \R^+ \times \R^+\times S^1)\\
((r,\theta),(r',\theta'))&\to& (\theta,(r,r',\alpha))=(\theta,(r,r',\theta-\theta'))\end{array}$$ 
 transforms $\xi_S$  into $\partial_\alpha.$ 
 Using this identification, the quotient manifold $(\R^+\times S^1\times \R^+ \times S^1)/S^1$ is just 
 $$P=\R^+\times \R^+ \times S^1$$ 
 and the reduced Poisson structure on $P$ is given by (see (\ref{wP}))
\begin{equation}\label{P}
\Pi_P(r,r',\alpha)=-\cos\alpha\;\partial_r\wedge \partial_{r'} + \frac{\sin\alpha}{r'}\partial_r\wedge \partial_{\alpha}-\frac{\sin\alpha}{r}\partial_{r'}\wedge \partial_{\alpha},
\end{equation}
where $(r,r',\alpha)$  are local coordinates on  $\R^+ \times \R^+\times S^1.$ 

\item{\bf The reduced Hamiltonian function:}
The reduced Hamiltonian function is 
\begin{equation}\label{HS1}
H^{S^1}(r,r',\alpha)=\frac{1}{2}(r^2 + (r')^2).
\end{equation}

\item{\bf The reduced dynamics:} The corresponding Hamiltonian vector field  on $P$ is just 
$$X_{H^{S^1}}^{\{\cdot,\cdot\}_P}= r\cos\alpha\; \partial_{r'}-r'\cos\alpha\;\partial_{r}- \sin\alpha(\frac{r}{r'} -\frac{r'}{r})\partial_\alpha.$$

\item {\bf The scaling symmetry on the reduced space:}
The projection on $P$ of the scaling symmetry $\Delta_S$ is 
$$\Delta_P=\frac{1}{2} (r{\partial_r} + r'{\partial_{r'}}),$$ 
which  generates the scaling action $\phi^P:\R^+\times (\R^+ \times \R^+\times S^1) \to  (\R^+ \times \R^+\times S^1)$ given by 
$$\phi^P(s,(r,r',\theta))=(\sqrt{s}r,\sqrt{s}r',\theta).$$

\end{itemize}

Now, using Theorem~\ref{reductionPoisson}, we can further reduce again the system (second step) with this last scaling symmetry. We obtain: 

\begin{itemize}
\item {\bf The reduced space:} Consider the diffeomorphim
$$\begin{array}{rcl}\R^+\times \R^+\times S^1&\to& \R^+\times (\R^+\times S^1)\\[5pt](r,r',\alpha)&\to &(\rho,\rho',\sigma)=(r,\displaystyle\frac{r'}{r},\alpha)\end{array}$$ 
which  transforms the generator  $\Delta_P$ of the $\R^+$-action on $\R^+\times \R^+\times S^1$  into the vector field 
$$\frac{1}{2}\rho\partial_\rho$$
with $(\rho,\rho',\sigma)$ local coordinates on $\R^+\times (\R^+\times S^1).$

Thus, the space of orbits of the reduced $\R^+$-action may be identified with $$K=\R^+\times S^1$$ and, under this identification, the canonical projection is 
$${\bf p_P}: P=\R^+\times \R^+\times S^1\to K=\R^+\times S^1,\qquad {\bf p_P}(r,r',\alpha)=(\frac{r'}{r},\alpha).$$

The associated line bundle is trivial 
$$\pi_L: L:=\R\times \R^+\times S^1\to \R^+\times S^1\qquad \pi_{L}(t,\rho',\sigma)=(\rho',\sigma)$$
and therefore, we have a Jacobi bracket on the space of functions on $K.$ In the sequel we will describe this structure. 

The expression of the reduced Poisson structure on $P$  in terms of the new local coordinates $(\rho,\rho',\sigma)$ is (see (\ref{P}))

\begin{equation}\label{P2}
\Pi_P=-\frac{\cos\sigma}{\rho}\partial_\rho\wedge \partial_{\rho'} + \frac{\sin\sigma}{\rho\rho'}\partial_\rho\wedge \partial_{\sigma}-2\frac{\sin\sigma}{\rho^2}\partial_{\rho'}\wedge \partial_{\sigma}
\end{equation} 
Note that 
$${\mathcal L}_{\frac{1}{2}\rho{\partial_\rho}}\Pi_{P}=-\Pi_{P}.$$ 

Since the homogenous functions with respect to the vector field $\frac{1}{2}\rho\partial_\rho$ are of  the form $\rho^2\,h$, with $h\in C^\infty(\R^+\times S^1)$, then we have that 
$$\{\rho^2h,\rho^2h'\}_P=\frac{1}{2}\rho{\partial_\rho}\{\rho^2h,\rho^2h'\}_P,$$
for all $h,h'\in C^\infty(\R^+\times S^1).$
This implies that  the Jacobi bracket $\{\cdot,\cdot\}_K$ on the space of functions on $K$ satisfies 
$$\{\rho^2h,\rho^2h'\}_P=\rho^2\{h,h'\}_K,\;\;\; h,h'\in C^\infty(\R^+\times S^1).$$
 As a consequence (see (\ref{P2})),
$$\{h,h'\}_K=-2\cos\sigma(h\partial_{\rho'}h'-h'\partial_{\rho'}h) + 2\frac{\sin\sigma}{\rho'}(h\partial_\sigma h'-h'\partial_\sigma h) - 2\sin \sigma(\partial_{\rho'}h\partial_{\sigma}h'-\partial_{\sigma}h\partial_{\rho'}h').$$

Therefore, the corresponding Jacobi structure $(\Pi_K,E_K)$ is 
\begin{equation}\label{JC2}
\Pi_K=-2\sin \sigma \partial_{\rho'}\wedge \partial_{\sigma},\qquad E_K=-2\cos\sigma\partial_{\rho'} + 2\frac{\sin \sigma}{\rho'}\partial_{\sigma}.
\end{equation}

\item{\bf The reduced Hamiltonian function:} 
The Hamiltonian function $H^{S^1}$ (see (\ref{HS1})), in terms of the local coordinates $(\rho,\rho',\sigma),$ is

$$
H^{S^1}(\rho, \rho',\sigma)=\frac{\rho^2}{2}(1 + (\rho')^2).$$
Since 
$$\frac{1}{2}\rho{\partial_\rho H^{S^1}}=H^{S^1},$$ we deduce that 
 $H^{S^1}(\rho^2,\rho',\sigma)=\rho^2h_{H^{S^1}}(\rho',\sigma)$ and therefore 
$$h_{H^{S^1}}(\rho',\sigma)=\frac{1}{2} (1+(\rho')^2).$$

\item{\bf The reduced dynamics:}
The Hamiltonian vector field induced by the previous Jacobi structure and the function $h_{H^{S^1}}$ is 
\begin{equation}\label{12}
X_{h_{H^{S^1}}}^{\{\cdot,\cdot\}_K}=-i(dh_{H^{S^1}})\Pi_K-h_{H^{S^1}}E_K=(1+(\rho')^2)\cos\sigma\;\partial_{\rho'}+\frac{(\rho')^2-1}{\rho'}\sin\sigma\;\partial_\sigma,
\end{equation}
which is the ${\bf p_P}$-projection of $X_{H^{S^1}}^{\{\cdot,\cdot\}_P}$.
\end{itemize}}
\end{example}

	\begin{example}[\bf {Continuing Example \ref{ex:projective0}: The linear Hamiltonian system reduced first by a standard and then by a scaling symmetry}]\label{ex:linear}
	
{\rm
	Let $\Phi:G\times Q\to Q$ be a free and proper action of a Lie group  $G$ on a manifold $Q$. Denote by  $0_Q$ the zero section of the cotangent bundle $\tau^*_Q:T^*Q\to Q$ and by $T^*\Phi:G\times (T^*Q-0_Q)\to (T^*Q-0_Q)$ the restriction to $T^*Q-0_Q$ of the cotangent lift action, i.e.~the free and proper action given by 
	\begin{equation}\label{35'}
	(T^*\Phi)_g(\alpha_q)=(T_{\Phi_g(q)}\Phi_{g^{-1}})^*(\alpha_q), \;\; \forall g\in G \mbox{ and } \forall \alpha_q\in T^*_qQ-0_q.
	\end{equation} 
	 It is well-known that $(T^*\Phi)_g$ is a symplectomorphim with respect to the standard symplectic structure $\omega_Q$ on $T^*Q-0_Q.$
	
	Suppose that $Y\in {\mathfrak X}(Q)$ is a $G$-invariant vector field on $Q$, that is, 
	\begin{equation}\label{YG}
	T_q\Phi_g(Y(q))=Y(\Phi_g(q)),\;\; g\in G \mbox{ and } q\in Q.
	\end{equation}
	
Moreover, let $\phi:\R-\{0\}\times (T^*Q-0_Q)\to (T^*Q-0_Q)$ be the action  given by (\ref{action1}). 

A direct computation, using (\ref{35'}) and (\ref{YG}), shows that the fiberwise-linear  function $Y^\ell:T^*Q\to \R$ induced by $Y$ is $G$-invariant, i.e.
$$Y^\ell\circ(T^*\Phi)_g=Y^\ell.$$
The symplectic action $T^*\Phi$ is fiberwise linear. So, 
$$(T^*\Phi)_g\circ \phi_s=\phi_s\circ (T^*\Phi)_g, \mbox{ for } g\in G \mbox{ and } s\in \R-\{0\}.$$

Thus, the previous comments imply that the actions $T^*\Phi$ and $\phi$ are compatible and the conditions of Theorem~\ref{2reduccion} hold. 
Now, we will reduce the Hamiltonian symplectic system $(T^*Q-0_Q,\omega_Q,Y^\ell)$, first by $T^*\Phi $ and then by the scaling symmetry $\phi.$ 
The objets obtained after the $G$-reduction are: 
\begin{itemize}
\item {\bf The reduced space:} The restriction of the canonical projection $\tau_Q^*:T^*Q\to Q$ to $T^*Q-0_Q$ is $G$-equivariant and therefore it  induces a fibration
$$\tau_Q^G:P:=(T^*Q-0_Q)/G\to Q/G$$ 
which is just the restriction of the Atiyah bundle $\tau_Q^G:T^*Q/G\to Q/G$ to $T^*Q/G-O,$ with $O$ the zero section of this vector bundle. 
The Poisson bracket $\{\cdot,\cdot\}_P$ on the space of functions $C^\infty((T^*Q-0_Q)/G)$ is characterized by 
$\{f\circ  \wp,g\circ \wp\}_{\omega_Q}= \{f,g\}_P\circ \wp$
with $$\wp:(T^*Q-0_Q)\to (T^*Q-0_Q)/G$$ the quotient map. 

\item {\bf The reduced Hamiltonian function:} The $G$-invariant function $Y^\ell$  induces a function $(Y^\ell)^G:(T^*Q-0_Q)/G\to \R$ such that 
$$(Y^\ell)^G(\wp(\alpha))=Y^\ell(\alpha).$$

\item {\bf The reduced dynamics:} The Hamiltonian vector field  $X_{Y\ell}^{\omega_Q}$ is $\wp$-projectable and its projection is just $$X_{(Y^\ell)^G}^{\{\cdot,\cdot\}_P}=\{\cdot,(Y^\ell)^G\}_P.$$

\item{\bf The scaling symmetry on the reduced space:} The scaling symmetry $\phi:(\R-\{0\})\times (T^*Q-0_Q)\to (T^*Q-0_Q)$ induces a scaling symmetry $\phi^G:(\R-\{0\})\times P\to P$ for the reduced Poisson Hamiltonian system $(P,\{\cdot,\cdot\}_P,(Y^\ell)^G)$ which is given by 
$$\phi^G(s,\wp(\alpha))=\wp(s\alpha),\mbox{ for } s\in \R-\{0\} \mbox{ and } \alpha\in T^*Q-0_Q.$$
\end{itemize}

Now, we will apply the second reduction step to the Poisson Hamiltonian  system $(P=(T^*Q-0_{Q})/G,\Pi_P, (Y^\ell)^G).$ with respect to the scaling symmetry $\phi^G:(\R-\{0\})\times P\to P.$ The reduced objects in this second reduction are:
\begin{itemize}
\item {\bf The reduced space:}
In this case, the reduced space is the projective bundle ${\Bbb P}(T^*Q/G)=((T^*Q-0_{Q})/G)/(\R-\{0\})$ of the vector bundle $(\tau^*_Q)^G: (T^*Q-0_G)/G\to Q/G$ (see Remark \ref{projPV}). 
\item {\bf The reduced Hamiltonian function:}
Denote by $\pi_L:L\to {\Bbb P}(T^*Q/G)$ the line bundle  associated with ${\bf p_P}:(T^*Q-0_{Q})/G\to {\Bbb P}(T^*Q/G).$ 
The section of the dual bundle $\pi_{L^*}:L^*\to {\Bbb P}(T^*Q/G)$ induced by the homogeneous function  $(Y^\ell)^G\in C^\infty((T^*Q-0_Q)/G)$ is the reduced Hamiltonian function. 

\item{\bf The reduced dynamics:} The Hamiltonian vector field $X_{(Y^\ell)^G}^{\{\cdot,\cdot\}_P}$ is ${\bf p_P}$-projectable and it determines the final reduced dynamics. 
\end{itemize}

\noindent {\bf The particular case of a Lie group.}  In what follows, we will show the previous  reduction process 
in the particular case when the initial manifold $Q$ is a Lie group $G.$ 
In such a case, one may use the left trivialization of the cotangent bundle $T^*G$ in order to identify $T^*G$ 
with the product manifold $G\times {\mathfrak g}^*$, where $({\mathfrak g},[\cdot,\cdot]_{\mathfrak g})$ is the Lie algebra of $G$, 
in such a way that the canonical projection $\tau_G^*:T^*G\to G$ is just the first projection $p_1:G\times {\mathfrak g}^*\to G.$ 
The left action  $\Phi:G\times G\to G$ on $G$ is the one defined by the group operation of $G$. We take the left invariant  
vector field $Y=\overleftarrow{\xi}$ on $G$ induced by  an element $\xi$ of ${\mathfrak g}.$ 
In the  first reduction with the cotangent lift of $\Phi$, the reduced space is $(T^*G-0_G)/G\cong {\mathfrak g}^*-\{0\}$ and   
the reduced function induced by $Y$ is the restriction to ${\mathfrak g}^*-\{0\}$ of the linear map ${\xi}^\ell$ associated with $\xi\in {\mathfrak g},$ i.e.
$$ {\xi}^\ell:{\mathfrak g}^*-\{0\}\to \R,\quad {\xi}^\ell(\alpha)=\alpha(\xi).$$
On the other hand,  the Lie-Poisson bracket $\{\cdot,\cdot\}_{{\mathfrak g}^*}$ on $(T^*G-0_G)/G\cong {\mathfrak g}^*-\{0\}$ is characterized by 
$$\{{\xi_1}^\ell,{\xi_2}^\ell\}_{{\mathfrak g}^*}=-{[\xi_1,\xi_2]}^\ell_{{\mathfrak g}},
\quad\mbox{ for all $\xi_1,\xi_2\in {\mathfrak g}.$}$$

 The scaling symmetry  on ${\mathfrak g}^*-\{0\}$ is just 
\begin{equation}\label{2sc}
\phi^G:(\R-\{0\})\times ({\mathfrak g}^*-\{0\})\to ({\mathfrak g}^*-\{0\}), \qquad (s,\alpha)\to s\alpha.
\end{equation} 

 Now, we apply the second reduction step to the (Lie)-Poisson Hamiltonian system $({\mathfrak g}^*-\{0\},\{\cdot ,\cdot\}_{{\mathfrak g}^*},\xi^\ell)$, 
 with respect to the scaling symmetry $\phi^G.$ In this case, the reduced space is the projective space ${\Bbb P}{\mathfrak g}^*$. 
 The corresponding line bundle $\pi_L:L:=({\mathfrak g}^*-\{0\}\times \R)/(\R-\{0\})\to {\Bbb P}{\mathfrak g}^*$ is defined by the action 
$$\widetilde\phi^G: (\R-\{0\})\times (({\mathfrak g}^*-\{0\})\times \R)\to ({\mathfrak g}^*-\{0\})\times \R,\qquad \widetilde\phi^G_s(\alpha,t)=(s\alpha,\frac{t}{s}).$$
The section of  the dual line bundle $\pi_{L^*}:L^*\to {\Bbb P}{\mathfrak g}^*$  associated with the linear map ${\xi}^\ell: {\mathfrak g}^*-\{0\}\to \R$ is 
$$h_{\xi}(p(\alpha))([(\alpha,t)])=t\alpha(\xi),$$
with $[(\alpha,t)]\in L,$ where $p:({\mathfrak g}^*-0)\to {\Bbb P}{\mathfrak g}^*$ is the quotient projection.

The Kirillov bracket on the projective space ${\Bbb P}{\mathfrak g}^*$ is characterized by 
$$
\begin{array}{rcl}[h_{\xi_1} ,h_{\xi_2}]_{{\Bbb P}{\mathfrak g}^*}(p(\alpha))([(\alpha,t)])&=&
-h_{\{\xi_1^\ell, \xi_2^\ell\}_{{\mathfrak g}^*}}(p(\alpha))([(\alpha,t)])=-t\{\xi_1^\ell, \xi_2^\ell\}_{{\mathfrak g}^*}(\alpha)\\[5pt]&=&
t\alpha([\xi_1,\xi_2]_{{\mathfrak g}})=h_{[\xi_1,\xi_2]_{{\mathfrak g}}}(p(\alpha))([(\alpha,t)]).
\end{array}$$

This structure on the line bundle $L\to {\Bbb P}{\mathfrak g}^*$ may be considered as the Kirillov version of the Lie-Poisson structure on ${\mathfrak g}^*$ and for this reason we will use the terminology {\it the  Lie-Kirillov structure on ${\Bbb P}{\mathfrak g}^*.$}

The reduced dynamics is determined by the $p$-projection of the Lie-Poisson Hamiltonian vector field associated with
the  linear function $\xi^\ell\in C^\infty({\mathfrak g}^*-\{0\})$, that is,
$$X_{{\xi}^\ell}^{\{\cdot,\cdot\}_{{\mathfrak g}^*}}=\{\cdot, {\xi}^\ell\}_{{\mathfrak g}^*}.$$

Note however, that this $p$-projection of $X_{\xi^\ell}^{\{\cdot,\cdot\}_{{\mathfrak g}^*}}$ is just the vector field 
$X_{h_\xi}\in {\mathfrak X}({\Bbb P}{\mathfrak g}^*)$, which is locally characterized by (\ref{21'}). 

	}

\end{example}

\section{Reduction of symplectic Hamiltonian systems using first the scaling symmetry and then the standard symmetries}\label{Section4}

As in the previous section, we have a  symplectic Hamiltonian system $(S,\omega,H)$  
with  a scaling symmetry $\phi^S:\R^\times \times S\to S$ and a symplectic $G$-symmetry $\Phi^S:G\times S\to S$ which are compatible. 
In what follows we describe the reduction process of the system $(S,\omega,H)$ in two steps, but in the following  order: 
the first reduction is obtained by the scaling symmetry and the second step is done using the standard symmetry.

First of all, we will show a reduction process for Kirillov structures  in the presence of a standard symmetry. 

\subsection{Reduction of Kirillov structures by standard symmetries}
Let $\pi_L:L\to K$ be a real line vector bundle with a Kirillov bracket $$[\cdot,\cdot]_{L^*}:\Gamma(L^*)\times \Gamma(L^*)\to \Gamma(L^*)$$
 on the space of the sections $\Gamma(L^*)$ of the dual vector bundle $\pi_{L^*}:L^*\to K$ of $\pi_L.$
Denote  by $0_L$  the zero section of $\pi_L$ and by $\phi^{L-0_L}:\R^\times \times (L-0_L)\to (L-0_L)$ the $\R^\times$-action 
associated with the principal bundle $p_{L-0_L}:(L-0_L)\to K$  whose line bundle is $\pi_L$ (see Appendix \ref{A}).
 
We suppose that $(\Phi^L:G\times L\to L,\Phi^K:G\times K\to K)$ is a representation  of a Lie group $G$ on the vector bundle $\pi_L:L\to K.$  
This means that $(\Phi^L_g,\Phi^K_g)$ is a vector bundle isomorphism for every $g\in G.$ So, we have a dual representation 
$(\Phi^{L^*}:G\times L^*\to L^*,\Phi^K:G\times K\to K)$ on the dual vector bundle $\pi_{L^*}:L^*\to K.$ Here, $\Phi^{L^*}:G\times L^*\to L^*$ 
is the representation of $G$ on $L^*$ induced by $\Phi^{L},$ 
given by 
$$\langle\Phi_g^{L^*}(\alpha),x\rangle=\langle\alpha,\Phi^L_{g^{-1}}(x)\rangle, \quad \mbox{ for all }\alpha\in L^* \mbox{ and } x\in L.$$

Note that 
$\pi_{L}\circ \Phi_g^{L}=\Phi_g^K\circ \pi_{L},$ which  implies that $\pi_{L^*}\circ \Phi_g^{L^*}=\Phi_g^K\circ \pi_{L^*},$ for all $g\in G.$

\begin{definition}
If the local Lie algebra structure $[\cdot,\cdot]_{L^*}$ is closed for $G$-equivariant sections of $\Gamma(L^*)$, 
we say that the representation $(\Phi^L:G\times L\to L,\Phi^K:G\times K\to K)$ {\it is compatible with the Kirillov structure.}
\end{definition}

We recall that a section $h:K\to L^*$ is \emph{$G$-equivariant} if 
$$\Phi^{L^*}_g\circ h=h\circ \Phi_g^{K},\mbox{ for all }g\in G.$$

On the other hand, since the principal bundle associated with $\pi_L$ is the   restriction  $p_{L-0_L}:L-0_L\to K$ of $\pi_L$ to $L-0_L$ (see Appendix \ref{A}), we deduce that 
\begin{equation}\label{pac}
\Phi_g^K\circ p_{L-0_L}=p_{L-0_L}\circ \Phi_g^L.
\end{equation}

In what follows, we suppose that the orbit space $K/G$ of the action $\Phi^K$ of $G$ on $K$ is a smooth quotient manifold. 
As a consequence, the orbit space $L/G$ is a real line bundle over $K/G$ whose fibers are isomorphic to the fibers of $\pi_L:L\to K.$ 

Denote by $0_{L/G}$  the zero section of the line bundle  $\pi_{L/G}:L/G\to K/G$.  The $\R^\times$-principal bundle  
$$p_{L/G-0_{L/G}}: (L/G-0_{L/G})\cong (L-0_L)/G\to K/G$$
 associated with $\pi_{L/G}$
is deduced from the $G$-equivariant principal bundle $p_{L-0_L}:(L-0_L)\to K.$

Moreover, the principal actions $\phi^{L-0_{L}}$ and $\phi^{(L-0_{L})/G}$  of $\R^\times$ on ${L-0_{L}}$ and ${(L-0_{L})/G},$ respectively, are related by 
\begin{equation}\label{tp}
 \wp_{{L-0_{L}}}\circ\phi_s^{L-0_{L}}=\phi_s^{{(L-0_{L})/G}}\circ \wp_{{L-0_{L}}}, \quad \mbox{ for }s\in \R^\times 
 \end{equation}
where  $\wp_{{L-0_{L}}}:{{L-0_{L}}}\to {({L-0_{L}})}/G$ is the quotient map. 

On the other hand, the dual vector bundle $\pi_{L/G}^*:(L/G)^*\to K/G$ is isomorphic to the line bundle $\pi_{L^*/G}:L^*/G\to K/G$ deduced from the $G$-equivariant  dual vector bundle $\pi_{L^*}:L^*\to K$ of $\pi_L$ for the pair of actions $(\Phi^{L^*}, \Phi^K).$ The following diagram summarizes the previous comments 

\[
\xymatrix{
L^*\ar@(ul,u)[]^{\Phi^{L^*}}\ar[d]^{\wp_{L^*}}\ar[rr]^{\pi_{L^*}}&&K\ar@(ul,u)[]^{\Phi^K}\ar[d]^{\wp_K}\\
L^*/G\ar[rr]^{\pi_{L^*/G}}&&K/G}
\]

Now we can prove the following general result that will be used in the following.
\begin{theorem}\label{Kirillov}\rm
Let  $[\cdot,\cdot]_{L^*}:\Gamma(L^*)\times \Gamma(L^*)\to \Gamma(L^*)$ be  a  Kirillov structure on the real line bundle $\pi_L:L\to K.$ 
Suppose that  $(\Phi^L,\Phi^K)$ is a compatible representation of $G$ on $L.$ Then: 
\begin{enumerate}
\item  There is a one-to-one correspondence between $G$-equivariant sections $h:K\to L^*$ of $\pi_{L^*}:L^* \to K$ with respect to $(\Phi^{L^*},\Phi^K)$ 
and sections ${h}^G:K/G\to L^*/G$ of the line bundle $\pi_{L^*/G}:L^*/G\to K/G.$
\item On the space of sections of $\pi_{L^*/G}:L^*/G\to K/G$ there is a Kirillov structure $[\cdot,\cdot]_{L^*/G}$, characterized by 
$$[h_1^G,h_2^G]_{L^*/G}=([h_1,h_2]_{L^*})^G,$$
for all $G$-equivariant sections $h_1,h_2$ of $\pi_{L^*}.$

\item If $h:K\to L^*$ is a $G$-equivariant section of $\pi_{L^*},$ then  the symbol $X_h^{[\cdot,\cdot]_{L^*}}\in {\mathfrak X}(K)$  associated with the derivation $[\cdot, h]_{L^*}$ 
is $G$-invariant with respect to $\Phi^K.$ Moreover, if $\wp_K:K\to K/G$ is the quotient map, the $\wp_K$-projection of $X_h^{[\cdot,\cdot]_{L^*}}\in {\mathfrak X}(K)$  is the symbol $X_{h^G}^{[\cdot,\cdot]_{L^*/G}}\in {\mathfrak X}(K/G)$ of the derivation $[\cdot,h^G]_{L^*/G}$.
\end{enumerate}
\end{theorem}
\begin{proof}
From the general theory of representations of Lie groups, we have that there is a one-to-one correspondence between $G$-equivariant sections 
$h:K\to L^*$ of $\pi_{L^*}:L^* \to K$ with respect to $(\Phi^{L^{*}},\Phi^{K})$ and sections ${h}^G:K/G\to L^*/G$ of the line bundle $\pi_{L^*/G}:L^*/G\to K/G$ such that 
$$h^G({\wp_K}(x))={\wp_{L^*}}(h(x)),\;\; \mbox{ for all } x\in K,$$
where ${\wp_{L^*}}:L^*\to L^*/G$ is the quotient  map.
Thus, we can induce a  bracket $[\cdot,\cdot]_{L^*/G}:\Gamma(L^*/G)\times \Gamma(L^*/G)\to \Gamma(L^*/G)$ characterized by 
\begin{equation}\label{hG}
[h_1^G,h_2^G]_{L^*/G}=([h_1,h_2]_{L^*})^G,
\end{equation}
where $h_1,h_2:L^*\to K$ are $G$-equivariant sections of $\pi_{L^*}:L^*\to K.$

It is clear that $[\cdot,\cdot]_{L^*/G}$ is a Lie algebra structure. On the other hand, if $f\in C^\infty(K/G)$, then 
$$
[(f\circ \wp_K)h_1,h_2]_{L^*}=(f\circ \wp_K)[h_1,h_2]_{L^*} + X_{h_2}^{[\cdot,\cdot]_{L^*}}(f\circ \wp_K)h_1,$$
for all $h_1,h_2\in \Gamma(L^*).$ 

Now, by hypothesis, the sections $[(f\circ \wp_K)h_1,h_2]_{L^*}$ and $(f\circ \wp_K)[h_1,h_2]_{L^*}$ are $G$-equivariant. 
Thus, $$X_{h_2}^{[\cdot,\cdot]_{L^*}}(f\circ \wp_K)h_1$$ 
is $G$-equivariant too, which implies that the function $X_{h_2}^{[\cdot,\cdot]_{L^*}}(f\circ \wp_K)$ is $\wp_K$-basic. 

So, we have proved that the vector field 
$X_{h_2}^{[\cdot,\cdot]_{L^*}}$ is $\wp_K$-projectable over a vector field $X_{h_2^G}^{[\cdot,\cdot]_{L^*/G}}(f\circ \wp_K)$ on $K/G$ and, in addition, 
$$\begin{array}{rcl}
[fh_1^G,h_2^G]_{L^*/G}&=&f[h_1^G,h_2^G]_{L^*/G} + X_{h^G_2}^{[\cdot,\cdot]_{L^*/G}}(f)h_1^G.
\end{array}$$

Therefore, $[\cdot,h_2^G]_{L^*/G}$  is a derivation and its symbol is just the ${\wp_K}$-projection of the symbol of $[\cdot, h_2]_{L^*}.$ This finishes the proof of the theorem. 
\end{proof}

The following diagram summarizes this reduction process 

\[
\xymatrix{
L^*\ar@(ul,u)[]^{G}\ar[d]^{\wp_{L^*}}\ar[rr]^{\pi_{L^*}}&&(K,\ar@(ul,u)[]^{G}X_{h}^{[\cdot,\cdot]_{L^*}})\ar[ll]<1ex>^h\ar[d]^{\wp_K}\\
L^*/G\ar[rr]^{\pi_{L^*/G}}&&(K/G, X_{h^G}^{[\cdot,\cdot]_{L^*/G}})\ar[ll]<1ex>^{h^G}}
\]

\begin{remark} When the real line bundle $\pi_L:L\to K$ is trivial, the previous theorem is just  the reduction process of Jacobi manifolds given in \cite{Nunes}. 
\end{remark}

\subsection{The first step: Reduction by a scaling symmetry}
Now,  we start with the scaling reduction process of the symplectic Hamiltonian system $(S,\omega,H).$ In this case we have (see Section \ref{ssp}):
\begin{itemize}
\item 
The reduced space $C=S/\R^\times$ admits a contact distribution ${\mathcal D}.$ 
\item The principal bundle ${\bf p_S}:S\to C$ is isomorphic to the restriction of $\pi_{{\mathcal D}^o}: {\mathcal D}^o\to C$ to $({\mathcal D}^o-0_C),$ 
where ${\mathcal D}^o$ is the annihilator of ${\mathcal D}$ and $0_C$ is its  zero section. Therefore, the associated real line bundle, 
under this isomorphism, is $\pi_{{\mathcal D}^o}:{\mathcal D}^o\to C.$ Moreover, there is a one-to-one correspondence between the sections $h:C\to ({\mathcal D}^o)^*$  
of the dual vector bundle of $\pi_{{\mathcal D}^o}$ and  the homogeneous functions $H_h:S\to \R$ on the symplectic manifold $S$. 
\item On the space $\Gamma(({\mathcal D}^o)^*)$ of the sections of the dual vector bundle  of $\pi_{{\mathcal D}^o},$ we have a Kirillov bracket $$[\cdot,\cdot]_{({\mathcal D}^o)^*}:\Gamma(({\mathcal D}^o)^*)\times \Gamma(({\mathcal D}^o)^*)\to \Gamma(({\mathcal D}^o)^*)$$
such that 
$$H_{[h_1,h_2]_{({\mathcal D}^o)^*}}=-\{H_{h_1},H_{h_2}\}_S,$$
for all $h_1,h_2\in \Gamma(({\mathcal D}^o)^*),$ where $\{\cdot,\cdot\}_S$ is the Poisson bracket associated with the symplectic structure on~$S.$ 

\item The Hamiltonian vector field $X_H^\omega\in {\mathfrak X}(S)$ of $H$ with respect to the symplectic structure  $\omega$ is ${\bf p_S}$-projectable on $C$ 
and its projection is   the symbol $X_{h_H}^{[\cdot,\cdot]_{({\mathcal D}^o)^*}}\in {\mathfrak X}(C)$ of the derivation $[\cdot,h_H]_{({\mathcal D}^o)^*}.$

\item A $G$-action on $C$. In fact, the relation $\Phi^S_g\circ \phi^S_s=\phi^S_s\circ \Phi^S_g$, for all $g\in G$ and $s\in \R^\times$,   
implies that   $\Phi_g^S:S\to S$ is $\R^\times$-equivariant and, therefore, it induces a principal action $\Phi^C:G\times C\to C$ such that 
\begin{equation}\label{aC}
\Phi^C_g\circ {\bf p_S}={\bf p_S}\circ \Phi_g^S.
\end{equation} Moreover, 
 \begin{equation}\label{inv-De}
 T\Phi^S_g\circ \Delta=\Delta\circ \Phi^S_g
  \end{equation}
  with $\Delta$ the infinitesimal generator of the scaling symmetry $\phi^S.$ Using this relation and that $(\Phi^S_g)^*\omega=\omega,$ 
  we conclude the $G$-invariance of the $1$-form $\lambda=-i_{\Delta}\omega,$ i.e.
 \begin{equation}\label{inv-lambda}
 (\Phi_g^S)^*(\lambda)= -(\Phi_g^S)^*(i_{\Delta}\omega)=-i_{\Delta}\omega=\lambda.
 \end{equation}
Therefore,
 \begin{equation}\label{lambda0}
 T\Phi_g^S(\left\langle\lambda\right\rangle^o)=\left\langle\lambda\circ \Phi^S_g\right\rangle^o,\mbox{ for all }g\in G,
 \end{equation}
 where $T\Phi^S:G\times TS\to TS$ is the tangent lift of the action of $\Phi^S$.   
 In other words, $\tilde{\mathcal D}=\left\langle\lambda\right\rangle^o$ is a $G$-invariant distribution. 
 So, since $\Phi_g^S\circ \phi^S_s=\phi^S_s\circ \Phi_g^S,$  we deduce that the contact distribution ${\mathcal D}=T{\bf p_S}(\tilde{\mathcal D})$ is $G$-invariant, i.e.
 $$
 T\Phi_g^C({\mathcal D})={\mathcal D}.
$$
This implies that the cotangent lift $T^*\Phi^C$ of the action $\Phi^C$ preserves the annihilator ${\mathcal D}^o$ of the contact distribution. 
Therefore, we have a representation $(\Phi^{{\mathcal D}^o}:=(T^*\Phi^C)_{|{\mathcal D}^o}, \Phi^C)$ of $G$ on the real line bundle $\pi_{{\mathcal D}^o}:{\mathcal D}^o\to C.$ 

\end{itemize}

\subsection{The second step: Reduction by standard symmetries}
Now, we apply the second reduction process with  the representation $(\Phi^{{\mathcal D}^o}:=(T^*\Phi^C)_{|{\mathcal D}^o}, \Phi^C)$. 
To do so, we will use Theorem~\ref{Kirillov} on the reduction of Kirillov structures.

\begin{theorem}\label{p2}\rm
Let $(S,\omega,H)$ be a symplectic Hamiltonian system with a scaling symmetry $\phi^S:\R^\times \times S\to S$, $G$ a  Lie group and  $\Phi^S:G\times S\to S$ a symplectic $G$-symmetry  which is  compatible with $\phi^S$. Then:

\begin{enumerate}
\item If $(C=S/\R^\times,{\mathcal D})$ is the contact manifold  induced by the scaling symmetry $\phi^S$, then we have  a representation $(\Phi^{{\mathcal D}^o}:G\times {\mathcal D}^o\to {\mathcal D}^o,{\Phi}^C:G\times C\to C)$ on the line bundle $\pi_{{\mathcal D}^o}:{\mathcal D}^o\to C$ such that  the corresponding quotient vector bundle $\pi_{{\mathcal D}^o/G}:{\mathcal D}^o/G\to C/G$ is a real line bundle. Moreover, there is a one-to-one correspondence between  the $G$-equivariant sections $h:C\to ({\mathcal D}^o)^*$ of the dual vector bundle of $\pi_{{\mathcal D}^o}:{\mathcal D}^o\to C$ and  sections  $h^G:C/G\to ({\mathcal D}^o)^*/G$ of the dual vector bundle of  $\pi_{{\mathcal D}^o/G}$.

%
%
%

\item  There is a  Kirillov bracket 
$[\cdot,\cdot]_{({\mathcal D}^o)^*/G}:\Gamma(({\mathcal D}^o)^*/G)\times \Gamma(({\mathcal D}^o)^*/G)\to \Gamma(({\mathcal D}^o)^*/G)$ 
on the space $\Gamma(({\mathcal D}^o)^*/G)$ of the sections of the dual vector bundle  $\pi_{({\mathcal D}^o)^*/G}:({\mathcal D}^o)^*/G\to C/G$,  
such that
$$([h_1h_2]_{({\mathcal D}^o)^*})^G=[h_1^G,h_2^G]_{({\mathcal D}^o)^*/G},$$
for $h_1,h_2\in\Gamma(({\mathcal D}^o)^*)$ $G$-invariant sections. 

\item  If $h_H:C\to ({\mathcal D}^o)^*$ is the section of $\pi_{({\mathcal D}^o)^*}:({\mathcal D}^o)^*\to C$ induced from $H$, 
the symbol $X_{h_H}^{[\cdot,\cdot]_{({\mathcal D}^o)^*}}\in{\mathfrak X}(C)$ of  the derivation $[\cdot, h_H]_{({\mathcal D}^o)^*}$
 is $G$-invariant and the corresponding vector field on $C/G$ is just 
 the symbol $X_{h_H^G}^{[\cdot,\cdot]_{({\mathcal D}^o)^*/G}}\in {\mathfrak X}(C/G)$ of the derivation $[\cdot, h_H^G]_{({\mathcal D}^o)^*/G}$. 
 Thus, if $\wp_C:C\to C/G$ is the quotient map, the Hamiltonian vector field $X_H^\omega\in {\mathfrak X}(S)$ of $H$ 
 with respect to the symplectic structure $\omega$ is $(\wp_C\circ {\bf p_S})$-projectable  on   
 $C/G$ and its projection is $X^{[\cdot,\cdot]_{({\mathcal D}^o)^*/G}}_{h_H^G}\in {\mathfrak X}(C/G).$ 
\end{enumerate}

\end{theorem}

\begin{proof}
  We have the representation $(\Phi^{{\mathcal D}^o}, \Phi^C)$ of $G$ on the real line bundle $\pi_{{\mathcal D}^o}:{\mathcal D}^o\to C$ defined previously. 

Now, we will prove that if $h_1,h_2:C\to ({\mathcal D}^o)^*$ are $G$-equivariant sections of $\pi_ {({\mathcal D}^o)^*}$,
  then the bracket  $[h_1,h_2]_{({\mathcal D}^o)^*}$ is also  $G$-equivariant. From (\ref{Hh}),  (\ref{hH}) (see Appendix \ref{A}) 
  and the commutation of the actions $\Phi^S$ and $\phi^S,$ we deduce   that  $h:C\to ({\mathcal D}^o)^*$ is a $G$-equivariant section 
  if and only if the corresponding homogeneous function $H_h:S\to \R$ is invariant with respect to the action $\Phi^S.$

So, if $h_1,h_2:C\to ({\mathcal D}^o)^*$ are $G$-equivariant, then $H_{h_1}$ and $H_{h_2}$ are $G$-invariant with respect to $\Phi^S$ and, since the action $\Phi^S$ is symplectic, we have that the function $\{H_{h_1},H_{h_2}\}_S$ is $G$-invariant. Therefore, 
$$
H_{[h_1,h_2]_{({\mathcal D}^o)^*}}(\Phi^S_g(x))=-\{H_{h_1},H_{h_2}\}_S(\Phi^S_g(x))=-\{H_{h_1},H_{h_2}\}_S(x)=H_{ [h_1,h_2]_{({\mathcal D}^o)^*}}(x),
$$
for all $x\in S$. In conclusion, $[h_1,h_2]_{({\mathcal D}^o)^*}$ is $G$-equivariant.

Now, applying Theorem \ref{Kirillov},  we deduce the result. 
\end{proof}

 The following diagram shows both reduction processes together 
 
  \[
\xymatrix{&\R&\\
(S,\omega,X_H^\omega)\ar[ur]^H\ar[rr]^{\kern-20pt {\bf p_S}}&&(C=S/\R^\times,{\mathcal D},X_{h_H}^{[\cdot,\cdot]_{({\mathcal D}^o)^*}})\ar[dd]^{\wp_C}\ar@/^/[dl]^{h_H}\\
&({\mathcal D}^o)^*\ar[dd]^{\wp_{({\mathcal D}^o)^*}}\ar@<1ex>[ur]^{\pi_{({\mathcal D}^o)^*}}&\\ &&(C/G,[\cdot,\cdot]_{({\mathcal D}^o)^*/G}, 
X_{h_H^G}^{[\cdot,\cdot]_{({\mathcal D}^o)^*/G}})\ar@/^/[dl]^{h_H^G}
\\&({\mathcal D}^o)^*/G\ar@<1ex>[ur]^{\pi_{({\mathcal D^o})^*/G}}&}
\]

Now we illustrate the reduction processes using the two examples considered above.

\begin{example}[{\bf Continuing Example~\ref{ex:2dho1}: The 2d harmonic oscillator reduced first by a scaling and then by a standard symmetry}]\label{ex:2dho2}\rm
We consider again the example of a 2-dimensional harmonic oscillator (see Examples~\ref{ex:STQho} and~\ref{ex:2dho1}). 
In Example ~\ref{ex:2dho1}  we have shown how to apply the reduction 
process by first using the standard symmetry and then the scaling symmetry. Now, we take the reverse order.

We recall that in this example we have:
\begin{enumerate}
\item A standard rotational $S^1$-symmetry, with infinitesimal generator $\xi_S=x\partial_{y}-y\partial_{x}+p_{x}\partial_{p_{y}}-p_{y}\partial_{p_{x}}$ where $(x,y,p_x,p_y)$ are coordinates on $S=T^*(\R^2-\{(0,0)\})$. Using the identification $\R^+\times S^1\times \R^+ \times S^1\cong T^*(\R^2-\{(0,0)\})-0_{\R^2-\{(0,0)\}}$  the local expression of $\xi_S$ is 
$$\xi_{S}=\partial_\theta +\partial_{\theta'},$$
where $(r,\theta,r',\theta')$ are polar coordinates on $\R^{+}\times S^1\times \R^{+} \times S^1.$ 
 \item A scaling $\R^{+}$-symmetry, with generator $$\Delta=\frac{1}{2} (r{\partial_r} + r'{\partial_{r'}}).$$
\end{enumerate}

As seen in Example~\ref{ex:STQho}, by applying first the scaling $\R^+$-symmetry, we obtain: 
\begin{itemize}
\item {\bf The reduced space:} It 
is  $\R^+\times S^1\times S^1$, with the quotient map 
$${\bf p}:\R^+\times S^1\times \R^+\times S^1\to \R^+\times S^1\times S^1, \quad {\bf p}(\rho,\theta,\rho',\theta')=(\rho',\theta,\theta').$$ 
The Jacobi structure   on $C=\R^+\times S^1\times S^1$ is given by (\ref{JC}).
\item {\bf The reduced Hamiltonian function:} The reduced Hamiltonian function is given by  $$H_{|\R^+\times S^1\times S^1}(\rho',\theta,\theta')=\displaystyle\frac{1}{2}((\rho')^2+1).$$ 
\item {\bf The reduced dynamics:} It is given by the vector field on $\R^+\times S^1\times S^1$ obtained by the ${\bf p}$-projection  
$${\bf p}_*(X^{\omega_Q}_H)=(1+(\rho')^{2})\cos(\theta-\theta')\partial_{\rho'}+\sin(\theta-\theta')(\frac{1}{\rho'}\partial_{\theta'}+\rho'\partial_{\theta}),$$
which is just the contact Hamiltonian vector field $X_{H_{|\R^+\times S^1\times S^1}}^{\{\cdot,\cdot\}_C}$ of the function $H_{|\R^+\times S^1\times S^1}$ 
with respect to the Jacobi structure on $C=\R^+\times S^1\times S^1$ described in (\ref{JC}).

\item 
{\bf The standard symmetry on the reduced space:} We may induce an $S^1$-action on the reduced space $\R^+\times S^1\times S^1$ whose infinitesimal generator is 
$$\xi_{\R^+\times S^1\times S^1}=\partial_\theta + \partial_{\theta'}.$$

\end{itemize}

Now, we apply the second step of the reduction process using this last symmetry, obtaining the reduction of the Kirillov structure by this standard symmetry. 
More precisely, the reduction of the Jacobi structure, because in this case the Kirillov line bundle is trivial. 

\begin{itemize}
\item {\bf The reduced space:}
We consider the diffeomorphism 
$$\begin{array}{rcl}
\R^+\times S^1\times S^1&\to& \R^+\times S^1\times S^1\\
(\rho',\theta,\theta')&\to&(\rho',\theta, \theta-\theta')
\end{array}
$$
which transforms $\xi_{\R^+\times S^1\times S^1}=\partial_\theta + \partial_{\theta'}$ into $\partial_\theta$. 
Therefore, the quotient space $(\R^+\times S^1\times S^1)/S^1$ may be identified with 
$$K=\R^+\times S^1\,,$$ 
so that $\wp_K: \R^+\times S^1\times S^1\to \R^+\times S^1$ is the map $\wp_K(\rho',\theta,\theta')=(\rho',\theta-\theta').$

In this case,  the line bundle associated with $\wp_K$ is trivial and we obtained a Jacobi structure. 
From (\ref{JC}) we deduce that the Jacobi structure on $K=\R^+\times S^1$ is 
\begin{equation}\label{JC3}
\Pi_K=-2\sin\sigma\partial_{\rho'}\wedge \partial_{\sigma},\qquad E_K=-2\cos\sigma\partial_{\rho'}+2\displaystyle\frac{\sin\sigma}{\rho'}\partial_{\sigma},
\end{equation}
with $(\rho',\sigma)$ polar coordinates on $\R^+\times S^1.$ Note that this Jacobi structure is just the one given in (\ref{JC2}). 

\item {\bf The reduced Hamiltonian function:} In this case,  the reduced Hamiltonian is 
$$H_{|\R^+\times S^1}(\rho',\sigma)=\displaystyle\frac{1}{2}((\rho')^2+1),$$ 
with $(\rho',\sigma)$ polar coordinates of $\R^+\times S^1.$ 

\item {\bf The reduced dynamics:} The reduced vector field is  the $\wp_K$-projection
 
$$
	(\wp_K)_{*}X_{H_{|\R^+\times S^1}}^{\{\cdot,\cdot\}_C}=(1+(\rho')^{2})\cos\sigma\,\partial_{\rho'}-\frac{1-(\rho')^{2}}{\rho'}\sin\sigma\,\partial_{\sigma}\,,
$$
which coincides precisely with the results obtained in Example~\ref{ex:2dho1},  using the reverse reduction process (see (\ref{12})).

\end{itemize}
	\end{example}

\begin{example}\label{ex:2dlinear2}{\bf Continuing Example \ref{ex:linear}: The linear Hamiltonian system reduced first by a scaling and then by a standard symmetry. }{\rm
We consider again the example of a  free and proper action $\Phi:G\times Q\to Q$  of a Lie group  $G$ on a manifold $Q$ with 
a  $G$-invariant vector field $Y\in {\mathfrak X}(Q).$ Then, we have two symmetries on $T^*Q-0_Q$: 
\begin{itemize}
\item The restriction $T^*\Phi:G\times (T^*Q-0_Q)\to (T^*Q-0_Q)$ to $T^*Q-0_Q$ of the cotangent lift of the action on $Q.$ 
\item  The scaling action $\phi:\R-\{0\}\times (T^*Q-0_Q)\to (T^*Q-0_Q)$ given by (\ref{action1}). 
\end{itemize}

In  Example~\ref{ex:linear} we have shown how to apply the reduction 
process by first using the standard symmetry and then the scaling symmetry. Now, we take the reverse order.

As seen in Example~\ref{ex:projective0}, by using first the scaling symmetry, we obtain the following reduced objects: 

\begin{itemize}
\item{\bf The reduced space:} It is the projective cotangent bundle
${\Bbb P}(T^*Q).$ Let  ${\mathcal D}$ be the contact distribution on ${\Bbb P}(T^*Q)$ such that ${\bf p}:{\mathcal D}^o-0_Q\to {\Bbb P}(T^*Q)$ 
is a principal bundle with real line bundle $\pi_{{\mathcal D}^o}: {\mathcal D}^o\to {\Bbb P}(T^*Q)$. 
The Kirillov bracket on the sections of $\pi_{({\mathcal D}^o)^*}: ({\mathcal D}^o)^*\to {\Bbb P}(T^*Q)$ satisfies 
$$[h_{X^\ell}, h_{Z^\ell}]_{({\mathcal D}^o)^*}=-h_{[X,Z]^\ell},$$
for all $X,Z\in {\mathfrak X}(Q)$. 

\item{\bf The reduced Hamiltonian section:} It is defined locally by (\ref{ell}). 

\item {\bf The reduced dynamics:} The Hamiltonian vector field $X^{\omega_Q}_{Y^\ell}\in {\mathfrak X}(T^*Q-0_Q)$ 
is ${\bf p}$-projectable and its projection is the symbol of the derivation $[\cdot, h_{Y^\ell}]_{({\mathcal D}^o)^*}.$

\item {\bf The standard symmetry on the reduced space:} The action is defined by 
$$G\times {\Bbb P}(T^*Q) \to {\Bbb P}(T^*Q),\qquad (g,{\bf p}(\alpha))\to {\bf p}((T^*\Phi)_g(\alpha)).$$
\end{itemize}

Now, we can consider the second step of the reduction process. 
The standard symmetry on the reduced space satisfies the conditions of Theorem \ref{Kirillov}, and therefore, we have
\begin{itemize}
\item {\bf The reduced space:} In this case,  the reduced space is the quotient space ${\Bbb P}(T^*Q)/G.$ 
Moreover, the projection ${\bf p:}{\mathcal D}^o-0\to {\Bbb P}(T^*Q)$ is $G$-invariant and it induces a reduced projection ${\bf p^G:}({\mathcal D}^o-0)/G\to {\Bbb P}(T^*Q)/G$. 
The real line bundle ${\bf \pi_{{\mathcal D}^o/G}:}L:={\mathcal D}^o/G\to K:={\Bbb P}(T^*Q)/G$ is deduced from the $G$-equivariant line bundle $\pi_{{\mathcal D}^o}.$ 
On the space of sections of the dual of this real bundle we have a Kirillov structure $[\cdot,\cdot]_{L^*}$ characterized by 
$$ [h_{X^\ell}^G, h_{Z^\ell}^G]_{L^*}=-h_{[X,Z]^\ell}^G,$$
 for $X,Z\in {\mathfrak X}(Q)$ $G$-invariant vector fields on $Q.$ 
 
 \item {\bf The reduced Hamiltonian section:} The section $h_{Y^\ell}$ of  $\pi_{({\mathcal D}^o)^*}: ({\mathcal D}^o)^*\to {\Bbb P}(T^*Q)$ is $G$-invariant and 
 therefore it induces a section $$h^G_{Y^\ell}:{\Bbb P}(T^*Q)/G\to ({\mathcal D}^o)^*/G.$$

\item {\bf The reduced dynamics:} The vector field ${\bf p}_*(X_{Y^\ell}^{\omega_Q})$ is $G$-invariant. Thus, it induces a vector field on ${\Bbb P}(T^*Q)/G,$ which is just the symbol of $[\cdot,h^G_{Y^\ell}]_{L^*}.$ 
\end{itemize}

\noindent {\bf The particular case of a Lie group.} When $Q=G$ is a Lie group,   
for the first reduction step with the scaling symmetry, we have (see Example \ref{ex:projective0}): 
\begin{itemize}
\item The reduced space is $G\times {\Bbb P}{\mathfrak g}^*.$ 
\item The contact structure is the distribution on $G\times {\Bbb P}{\mathfrak g}^*$ given by 
$${\mathcal D}_{(g,p(\mu))}=\left\langle(T_gL_{g^-1})^*(\mu)\right\rangle^o\times T_{p(\mu)}({\Bbb P}{\mathfrak g}^*)$$
for all $g\in G$ and $\mu\in {\mathfrak g}^*-\{0\}.$ 
\item 
The fiber of the real line bundle $\pi_{{\mathcal D}^o}: {{\mathcal D}^o}\to G\times {\Bbb P}{\mathfrak g}^*$  at $(g,p(\mu))\in G\times {\Bbb P}{\mathfrak g}^*$ is just 
$${{\mathcal D}^o}_{(g,p(\mu))}=\left\langle(T_gL_{g^-1})^*(\mu)\right\rangle.$$

\item The reduced Hamiltonian section of $\pi_{({\mathcal D}^o)^*}: ({{\mathcal D}^o})^*\to G\times {\Bbb P}{\mathfrak g}^*$ induced by the function  $Y^\ell$  is characterized by 
$$h_\xi(g,p(\mu))((T_gL_{g^{-1}})^*(\mu))=\mu(\xi),$$
with $g\in G$, $\mu\in {\mathfrak g}^*-\{0\}$, $\xi=Y(e)$ and  $p:{\mathfrak g}^*-\{0\}\to {\Bbb P}{\mathfrak g}^*$ the corresponding quotient map determined by the scaling symmetry on   
${\mathfrak g}^*-\{0\}.$ 
\item The reduced vector field after this reduction is $(Y, X_{h_\xi})\in {\mathfrak X}(G)\times {\mathfrak X}({\Bbb P}{\mathfrak g}^*),$ such that 
\begin{equation}\label{h}
X_{h_\xi}(f)\circ p=\{f\circ p, {\xi}^\ell\}_{{\mathfrak g}^*-\{0\}},
\end{equation}
which is the symbol of the derivation $[\cdot, h_\xi]_{({\mathcal D}^o)^*}.$ 
\end{itemize}

Now, if we perform the second reduction step  associated with the induced $G$-action 
$$G\times (G\times {\Bbb P}{\mathfrak g}^*)\to G\times {\Bbb P}{\mathfrak g}^*,\;\; (g', (g,p(\mu)))\to (gg',p(\mu)),$$
the corresponding reduced elements are: 

\begin{itemize}
\item The reduced space is the projective space ${\Bbb P}{\mathfrak g}^*$. 
\item The line vector bundle $\pi_L:L\to {\Bbb P}{\mathfrak g}^*$ is given by 
$$L_{p(\mu)}=\left\langle\mu\right\rangle,\;\; \mu\in {\mathfrak g}^*.$$
\item The reduced section of $\pi_{L^*}:L^*\to {\Bbb P}{\mathfrak g}^*$ is just 
$${h}^G_\xi(p(\mu))(t\mu)=t\mu(\xi).$$ 
\item The final reduced dynamics is the vector field $X_{h_\xi}$ on ${\Bbb P}{\mathfrak g}^*$ described in   (\ref{h}), 
which is the symbol of $[\cdot,{h}^G_\xi]_{L^*}$ and whose local expression is (\ref{21'}). 
\end{itemize}

	}\end{example}
	
So, also in this case, similarly to the two previous examples (see Examples~\ref{ex:2dho1},~\ref{ex:linear}, ~\ref{ex:2dho2} and~\ref{ex:2dlinear2}), 
both reduction processes give rise to the same reduced dynamics. This fact
motivates further analysis on the equivalence of the two reduction processes, which will be addressed in full generality in the following section.

\section{The equivalence of the two reduction processes}\label{sec:equiv}

Finally,  we will prove that both processes considered in Sections~\ref{section3} and~\ref{Section4} are equivalent. 
Let $(S,\omega,H)$ be a symplectic Hamiltonian system with a scaling symmetry $\phi^S:\R^\times \times S\to S$ and a symplectic $G$-symmetry 
$\Phi^S:G\times S\to S$ which are compatible,  $G$ being a Lie group. 
\begin{theorem}\label{equiv}{\rm
Under the previous conditions we have that:
\begin{enumerate}
\item  There exists  a real line bundle isomorphism $(\Psi,\psi)$   between the line bundles  $\pi_{ {L}}:L\to (S/G)/\R^\times$ and $\pi_{({\mathcal D}^o)/G}: ({\mathcal D}^o)/G\to (S/\R^\times)/G$

\medskip 

\hspace{80pt}
\xymatrix{ 
 {L}\ar[r]^\Psi\ar[d]^{\pi_{ {L}}} & \ar[d]^{{\pi_{{\mathcal D}^o/G}}}{\mathcal D}^o/G\\
(S/G)/\R^\times\ar[r]^{ \psi}  & (S/\R^\times)/G}

  \medskip 
  
 \item 
 The sections  $h_H^G\in \Gamma(({\mathcal D}^o)^*/G)$  and $h_{H^G}\in \Gamma( {L}^*)$ induced by the Hamiltonian function $H:S\to \R$  and obtained in 
 Theorem \ref{p2} and  Theorem \ref{2reduccion} respectively, are related as follows 
\begin{equation}\label{HG}
h_{H^G}=\Psi^*\circ h_H^G \circ \psi,
\end{equation}
where  $\Psi^*$ is the  dual isomorphism, between the line bundles $\pi_{({\mathcal D}^o)^*/G}$ and $\pi_{L^*}$, deduced from $\Psi.$ 

 \item The Kirillov structures  $[\cdot,\cdot]_{  L^*}$ and $[\cdot,\cdot]_{({\mathcal D}^o)^*/G}$ obtained in  Theorem \ref{2reduccion} and Theorem \ref{p2}  respectively, are isomorphic. 
 In fact, we have that  
 \begin{equation}\label{iso}
  [\Psi^*\circ {h_1^G}\circ \psi,\Psi^*\circ {h_2^G}\circ \psi]_{ {L}^*}=\Psi^*\circ [{h}_1^G,{h_2^G}]_{({\mathcal D}^o)^*/G}\circ \psi,
  \end{equation}
 for all ${h}_1, {h}_2$ $G$-invariant sections of the line bundle $\pi:{\mathcal D}^o\to S/\R^\times$.  

\item The vector fields $X_{h_{H^G}}^{[\cdot,\cdot]_{ {L}^*}}$ and  $X_{h_H^G}^{[\cdot,\cdot]_{({\mathcal D}^o)^*/G}}$ given in  
Theorem \ref{2reduccion} and Theorem \ref{p2} respectively,  are $\psi$-related, i.e.~the following diagram is commutative 
\[
\xymatrix{
(S/G)/\R^\times\ar[d]_{X_{h_{H^G}}^{[\cdot,\cdot]_{ {L}^*}}}\ar[rr]^{\psi}&&(S/\R^\times)/G\ar[d]^{X_{h_H^G}^{[\cdot,\cdot]_{({\mathcal D}^o)^*/G}}}\\
T((S/G)/\R^\times)\ar[rr]^{T\psi}&&T((S/\R^\times)/G)}
\]

\end{enumerate}}
 \end{theorem}
 
 \begin{proof}
  $(1)$ The diffeomorphism $\psi$ is just 
 \begin{equation}\label{psi}
 \psi: (S/G)/\R^\times \to (S/\R^\times)/G,\qquad \psi({\bf p_P}(\wp_S(x)))=\wp_C({\bf p}_S(x)),\quad \mbox{ for all }x\in S,
 \end{equation}
 that is, 
 $$
 \xymatrix{ 
 {S}\ar[r]^{Id_S}\ar[d]^{\wp_S} & \ar[d]^{{\bf p_S}}S\\
 {P=S/G}\ar[d]^{{\bf p_P}} & \ar[d]^{\wp_C}C=S/\R^\times\\
(S/G)/\R^\times\ar[r]^{ \psi}  & (S/\R^\times)/G}
$$

 We remark that this map is a diffeomorphism from the equality  $\Phi_g^S\circ\phi_s^S=\phi^S_s\circ \Phi^S_g.$ 
 Moreover,  the diffeomorphism $\Psi$ is characterized in this diagram 
 \begin{equation}\label{Psi}
 \xymatrix{ 
 { S\times \R}\ar[r]^{Id_{S\times \R}}\ar[d]^{ \wp_S\times Id_{\R}} & \ar[d]^{{\bf p_{S\times \R}}} S\times \R\\
 {P\times \R=S/G\times \R}\ar[d]^{{\bf p_{P\times \R}}} & \ar[d]^{\wp_{(S\times \R)/\R^\times}}(S\times \R)/\R^\times\\
L=((S/G)\times \R)/\R^\times\ar[r]^{ \kern-10pt\Psi}  & ((S\times \R)/\R^\times)/G\cong {\mathcal D}^o/G}
\end{equation}
Here ${\bf p_{P\times \R}}$ is the quotient map deduced from the action 
$$\R^\times \times (P\times \R)\to (P\times \R),\qquad (s,(\wp_S(x),t))\to (\wp_S(sx),\frac{t}{s}),$$
and 
${\bf p_{S\times\R}}$ the quotient map deduced from the action 
$$\R^\times \times (S\times\R)\to (S\times\R ),\qquad (s,(x,t))\to (sx,\frac{t}{s}).$$


  \medskip 
 
 (2) From (\ref{hH}) in Appendix \ref{A},  we have  
 $${h}_{H^G}({\bf p_{P}}(\wp_S(x)))({\bf p_{P\times\R }}(\wp_S(x),t))=tH^G(\wp_S(x))=tH(x),$$ 
 for $x\in S$ and $t\in \R.$ 

On the other hand,  using (\ref{psi}), the diagram (\ref{Psi}) and again (\ref{hH}) in Appendix \ref{A}, we obtain 
$$
\begin{array}{rcl}
(\Psi^*\circ h_H^G\circ  \psi)({\bf p_{P}}(\wp_S(x)))({\bf p_{P\times \R}}(\wp_S(x),t))
=h_H({\bf p_S}(x))({\bf p_{S\times \R}}(x,t))=tH(x).
\end{array}$$

\medskip

(3) If $h_1,h_2$ are $G$-invariant
sections of the line bundle $\pi_{({\mathcal D}^o)^*}: ({\mathcal D}^o)^*\to C=S/\R^+,$ then from (2) in Theorem~\ref{p2} and (\ref{HG}), we deduce 
$$
\begin{array}{rcl}
H_{\Psi^*\circ [h_1^G,h_2^G]_{({\mathcal D}^o)^*/G}\circ \psi}\circ \wp_S&=&
H_{\Psi^*\circ [h_1,h_2]_{({\mathcal D}^o)^*}^G\circ \psi}\circ \wp_S=H^G_{[h_1,h_2]_{({\mathcal D}^o)^*}}\circ \wp_S\\[5pt]&=&
H_{[h_1,h_2]_{({\mathcal D}^o)^*}}=-\{H_{h_1},H_{h_2}\}_S.
\end{array}$$

On the other hand, using $b)$ in Theorem~\ref{reductionPoisson}, (\ref{pc=cp}) and (\ref{HG}), we have 
$$
\begin{array}{rcl}
H_{[\Psi^*\circ {h}_1^G\circ\psi ,\Psi^*\circ {h}_2^G\circ \psi]_{ {L}^*}}\circ \wp_S&=&
-\{H_{\Psi^*\circ {h}_1^G\circ \psi},H_{\Psi^*\circ {h}_2^G\circ \psi}\}_{P}\circ \wp_S=-\{H_{\Psi^*\circ {h}_1^G\circ \psi}\circ \wp_S, H_{\Psi^*\circ{h}_2^G\circ \psi}\circ \wp_S\}_S\\[5pt]&=&
-\{H_{h_1}^G\circ \wp_S,H_{h_2}^G\circ \wp_S\}_S=-\{H_{h_1},H_{h_2}\}_S.
\end{array}
$$

Therefore, we have (\ref{iso}). 

\medskip

(4) We consider the section  $\Psi^*\circ h^G\circ \psi\in \Gamma( {L}^*),$ with $h$ a $G$-invariant section on $\pi_{({\mathcal D}^o)^*}$ and $f\in C^\infty((S/\R^\times)/G).$ From the properties of the Kirillov structure $[\cdot,\cdot]_{L^*},$ we have that  
$$
\begin{array}{rcl}\label{u1}
{[ (f\circ \psi)(\Psi^*\circ h^G\circ \psi),h_{H^G}]_{ {L}^*}}&=& (f\circ\psi)[\Psi^*\circ h^G\circ \psi, h_{H^G}]_{ {L}^*} + X_{h_{H^G}}^{[\cdot,\cdot]_{L^*}}(f\circ \psi)(\Psi^*\circ h^G\circ \psi).
\end{array}$$

On the other hand, using (\ref{HG}) and (\ref{iso}), we obtain 

$$[ (f\circ \psi)(\Psi^*\circ h^G\circ \psi),h_{H^G}]_{ {L}^*}=[ \Psi^*\circ(fh^G)\circ \psi , \Psi^*\circ h_{H}^G\circ \psi]_{L^*}=\Psi^*\circ[fh^G ,h_{H}^G]_{ ({{\mathcal D}^o)}^*/G}\circ \psi,$$

$$(f\circ\psi)[\Psi^*\circ h^G\circ \psi, h_{H^G}]_{ {L}^*}=\Psi^*\circ(f[h^G ,h_{H}^G]_{ ({{\mathcal D}^o)}^*/G})\circ \psi.$$

Replacing these relations in (\ref{u1}),  we have that 
\begin{equation}\label{u2}
\Psi^*\circ[fh^G ,h_{H}^G]_{ ({{\mathcal D}^o)}^*/G}\circ \psi=\Psi^*\circ(f[h^G ,h_{H}^G]_{ ({{\mathcal D}^o)}^*/G})\circ \psi + X_{h_{H^G}}^{[\cdot,\cdot]_{L^*}}(f\circ \psi)(\Psi^*\circ h^G\circ \psi).\end{equation}

However, we know that
\begin{equation}\label{u3}
[fh^G ,h_{H}^G]_{ ({{\mathcal D}^o)}^*/G}=f[h^G, h_H^G]_{({\mathcal D}^o)^*/G} + X_{h_{H}^G}^{[\cdot,\cdot]_{({\mathcal D}^o)^*/G}}h^G.
\end{equation}

Comparing (\ref{u2}) and (\ref{u3}), we conclude $(4).$

\end{proof}

Both reduction processes and the corresponding equivalence between them are  summarized in the following diagram
 \[
\xymatrix{&\R&\R&\\
(P=S/G,\{\cdot,\cdot\}_P,X^{\{\cdot,\cdot\}_P}_{H^G})  \ar[dd]^{{\bf p_{P}}}\ar[ur]^{H^G}&(S,\omega,X_H^\omega)\ar[l]^{\kern30pt\wp_S}\ar[ur]^H\ar[r]^{\kern-20pt {\bf p_S}}&(C=S/\R^\times,{\mathcal D},X_{h_H}^{[\cdot,\cdot]_{({\mathcal D}^o)^*}})\ar[dd]^{\wp_C}\ar@/^/[dl]^{h_H}\\
&({\mathcal D}^o)^*\ar@<1ex>[ur]^{\pi_{({\mathcal D}^o)^*}}&\\ ((S/G)/\R^\times,[\cdot,\cdot]_{  L^*}, X_{h_{H^G}}^{[\cdot,\cdot]_{  L^*}})\ar@{<-->}[rr]^\cong \ar@/^/[dr]^{\kern10pt h_{H^G}}&&(C/G,[\cdot,\cdot]_{({\mathcal D}^o)^*/G}, X_{h_H^G}^{[\cdot,\cdot]_{({\mathcal D}^o)^*/G}})\ar@/^/[dl]^{h_H^G}&
\\&\;\;\;\;\;\;{ L}^*\cong \ar@<1ex>[ul]^{ \pi_{ L^*}}({\mathcal D}^o)^*/G\ar@<1ex>[ur]^{\pi_{({\mathcal D}^o)^*/G}}&&}
\]

\section{Reconstruction process for scaling symmetries}\label{sec:reconstruction}

In this section we will study the inverse process of reduction:
the \textit{reconstruction process}. First, we shall introduce the
general involved ideas, for arbitrary dynamical systems and Lie groups,
and then we shall concentrate on the case of symplectic Hamiltonian
systems with scaling symmetries.

\subsection{The general context}

Let $M$ be a manifold, $X\in\mathfrak{X}\left(M\right)$ a vector
field on $M$ and $G$ a Lie group acting on $M$ by an action $\phi^{M}:G\times M\rightarrow M$
such that $X$ is $G$-invariant. Assume that $\phi^{M}$ defines
a principal fiber bundle $p_{M}:M\rightarrow M/G$. In such a case,
the $G$-invariance of $X$ ensures that there exists a vector field
$X^G\in\mathfrak{X}\left(M/G\right)$ such that $X^G\circ p_{M}=Tp_{M}\circ X$.
The question is: how can we get the integral curves of $X$ from those
of $X^G$? To do that, we can proceed as follows. If we want the integral
curve $\Gamma:\left(-\epsilon,\epsilon\right)\rightarrow M$ of $X$
such that $\Gamma\left(0\right)=x_{0}$, then:
\begin{enumerate}
\item consider the integral curve $\gamma:\left(-\epsilon,\epsilon\right)\rightarrow M/G$
of $X^G$ such that $\gamma\left(0\right)=p_{M}\left(x_{0}\right)$;
\item fix a principal connection $A:TM\rightarrow\mathfrak{g}$ for $p_{M}$
(where $\mathfrak{g}$ is the Lie algebra of $G$) and fix a curve
$\varphi:\left(-\epsilon,\epsilon\right)\rightarrow M$ such that $\varphi\left(0\right)=x_{0}$,
\begin{equation}
A\left(\varphi'\left(t\right)\right)=0\quad\textrm{and}\quad p_{M}\left(\varphi\left(t\right)\right)=\gamma\left(t\right)\label{2-2}
\end{equation}
(in other words, $t\to \varphi(t)$ is the horizontal lift of the curve $\gamma$ by the principal connection $A$);

\item and find the curve $g:\left(-\epsilon,\epsilon\right)\rightarrow G$
such that
\begin{equation}
g'\left(t\right)=T_eL_{g\left(t\right)}\left[A\left(X\left(\varphi\left(t\right)\right)\right)\right],\quad g\left(0\right)=e.\label{3-2}
\end{equation}
\end{enumerate}
From now on, we shall take $\epsilon$ small enough in order to fulfill
above conditions. Then, proceeding as in \cite{AM} (see pages 304-305), one may prove that 
\[
\Gamma\left(t\right)=\phi^{M}\left(g\left(t\right),\varphi\left(t\right)\right)
\]
is the curve we are looking for. The above three-step procedure is
usually known as \textit{reconstruction}. The steps $2$ and $3$
are known as the \textit{reconstruction problem} (see,  for example \cite{MMR}).

Clearly, such a procedure can be used for the standard as well as for the scaling symmetries. 
In the following, we shall focus on the latter, since the reconstruction process for scaling symmetries, as far as the authors know, has not been studied in the literature so far.

\subsection{Application to scaling symmetries and symplectic Hamiltonian systems}

Now, as in Section \ref{ssp}, let us suppose that we have a scaling
symmetry $\phi:\R^{\times}\times S\to S$ on a symplectic Hamiltonian
system $(S,\omega,H)$, with infinitesimal generator $\Delta$. Then, assuming
that ${\bf p_{S}}:S\rightarrow C=S/\R^{\times}$ is a principal bundle
(see the first part of Section  \ref{ssp}), 
\begin{itemize}
\item we have a contact distribution ${\mathcal{D}}$ on $C$ and a related
real line bundle $\pi_{{\mathcal{D}}^{o}}:{\mathcal{D}}^{o}\to C$ with a
Kirillov structure $[\cdot,\cdot]_{({\mathcal{D}}^{o})^{*}}$, 
\item and we can ensure that the Hamiltonian vector field $X_{H}^{\omega}\in{\mathfrak{X}}(S)$
of $H$ projects onto the symbol $X_{h_{H}}^{[\cdot,\cdot]_{({\mathcal{D}}^{o})^{*}}}\in{\mathfrak{X}}(C)$
of the derivation $[\cdot,h_{H}]_{({\mathcal{D}}^{o})^{*}}.$
\end{itemize}
Recall that $h_{H}:C\rightarrow({\mathcal{D}}^{o})^{*}$ denotes
the section of $\Gamma(({\mathcal{D}}^{o})^{*})$ related to the homogeneous function $H$. So, we are in
the situation of the previous  subsection, with $M=S$, $X=X_{H}^{\omega}$,
$G=\R^{\times}$, $\mathfrak{g}=\R$ and $X^G=X_{h_{H}}^{[\cdot,\cdot]_{({\mathcal{D}}^{o})^{*}}}$.
We shall apply the reconstruction procedure described above in this
particular context. 

\subsubsection{Existence of a flat connection}

There is a case in which solving the reconstruction problem is especially
simple (as we will show later). This case is when there is a non-vanishing homogeneous function $F:S\rightarrow\R^{\times}$. This kind of functions are 
called \textit{scaling functions}~\cite{BJS}.

In such a case  the map
\[
\left(F,{\bf p_{S}}\right):S\rightarrow\R^{\times}\times C
\]
is a diffeomorphism and defines a trivialization for ${\bf p_{S}}$. Its
inverse is given by
\[
\left(F,{\bf p_{S}}\right)^{-1}:\left(s,{\bf p_{S}}\left(x\right)\right)\in\R^{\times}\times C\rightarrow\phi\left(\frac{s}{F\left(x\right)},x\right)\in S,
\]
for all $s\in\R^{\times}$ and $x\in S$, and we have a global section
$\sigma:C\rightarrow S$ of ${\bf p_S}$ which takes the values
\begin{equation}\label{sigma}
\sigma\left(y\right)=\left(F,{\bf p_{S}}\right)^{-1}\left(1,y\right),\quad\forall y\in C.
\end{equation}

Conversely, if ${\bf p_S}:S\to C$ is trivial, i.e. $S\cong \R^\times \times C$ and ${\bf p_S}$ is the second projection,  
then the function $F:S\cong \R^\times \times C\to \R$ given by $F(s,x)=s$ is a non-vanishing homogenous function, i.e.~a scaling function. 

Therefore, the existence of a scaling function $F$ on $S$ is equivalent with the trivialization of the principal bundle ${\bf p_S}:S\to C.$ 
This fact guarantees the local existence of this kind of functions $F$  (see \cite{BJS}).

Moreover,  if $\Delta$ is the infinitesimal generator of $\phi$, since 
\[
dF\left(x\right)\left(\Delta\left(x\right)\right)=F\left(x\right)\neq0,\quad\forall x\in S,
\]
then  we have that 
$$TS=\left\langle \Delta\right\rangle\oplus \left \langle dF\right\rangle ^{o}.$$
So, the map $A:TS\rightarrow\R$, characterized by 
\begin{equation}
A\left(\Delta\left(x\right)\right)=\Delta(F)\left(x\right),\quad\forall x\in S,\label{ad}
\end{equation}
and
\begin{equation}
\ker A=\left\langle dF\right\rangle ^{o},\label{ka}
\end{equation}
is a principal \textit{flat connection} for ${\bf p_{S}}$ (because $\ker A$
is integrable). 

On the other hand, the $1$-form $\eta:=\sigma^*(\lambda)$ is a global generator
of ${\mathcal{D}}^{o}$ with $\lambda=-i_\Delta\omega,$ which makes $\pi_{{\mathcal{D}}^{o}}$ trivial. 
In fact, using (\ref{sigma}), we have that $\sigma\circ {\bf p_S}=\phi_{\frac{1}{F}}$. Therefore, 
$$ \left(\bf p_{S}\right)^{*}\eta=(\phi_{\frac{1}{F}})^*\lambda=(\phi\circ (\frac{1}{F},Id))^*\lambda.$$

Since $\lambda=-i_\Delta\omega, $ then $T_s^*\phi_x(\lambda(\phi(x,s))=0,$ for all $(s,x)\in \R^\times\times S.$ Thus, from the homogeneity of $\lambda,$ we have that 
$$\left((\bf p_{S}\right)^{*}\eta)(x)=(\frac{1}{F},Id)^*(0,(\phi_{\frac{1}{F(x)}})^*(\lambda(x))=\frac{1}{F(x)}\lambda(x).$$

In conclusion, we deduce that 
\begin{equation}
\left(\bf p_{S}\right)^{*}\eta=\displaystyle\frac{1}{F}\lambda.\label{psl}
\end{equation}
 This implies that ${\mathcal{D}}=\left\langle \eta\right\rangle ^{o}$, and $\eta$ is a contact $1$-form on $C.$ 
 
 The  one-to-one correspondence 
 between homogeneous functions $H:S\to \R$ on $S$ and functions $h_H:C\to \R$ on $C$ 
 (sections of the trivial line bundle $\pi_{({\mathcal D}^o)^*}$)  is defined by the relation 
 $$h_H\circ {\bf p_S}=\frac{1}{F}H.$$

Note that the function on $C$ associated with $F$ is just the constant function $1$.

The Jacobi bracket of  two functions $h_1,h_2$ on $C$ defined by the contact $1$-form $\eta$ is given by 
\begin{equation}
\left\{h_{1},h_{2}\right\}_C\circ {\bf p_{S}}=-\frac{1}{F}\left\{ F({h_1}\circ {\bf p_{S}}),F({h_2}\circ {\bf p_{S}})\right\} _{S}.\label{braH}
\end{equation}

The relation between the Hamiltonian vector field $X_{H}^{\omega}$  of $H$ with respect $\omega$ and the Hamiltonian vector field $X_{h_H}^{\eta}$ of $h_H$  
with respect to the contact structure $\eta$ is 
\[
T{\bf p_{S}}\circ X_{H}^{\omega}=X_{h_{H}}^{\eta}\circ {\bf p_{S}}.
\]

\begin{remark}
Since above equation is actually true for any homogeneous function
$H$, for $H=F$ we have that 
\begin{equation}
T{\bf p_{S}}\circ X_{F}^{\omega}=X_{1}^{\eta}\circ {\bf p_{S}}=\mathcal{R}\circ {\bf p_{S}},\label{pxxp}
\end{equation}
where $\mathcal{R}$ is the Reeb vector field of $\eta$.\bigskip{}
\end{remark}
Below, we shall use all these facts to address the reconstruction
problem for the system $\left(S,\omega,H\right)$ and the action~$\phi$. 

\subsubsection{Solving the reconstruction problem}

Suppose that we want to find the integral curve $\Gamma:\left(-\epsilon,\epsilon\right)\rightarrow S$
of $X_{H}^{\omega}$ such that $\Gamma\left(0\right)=x_{0}$. Following
the step $1$ of the reconstruction procedure, let us fix the integral
curve $\gamma:\left(-\epsilon,\epsilon\right)\rightarrow C=S/\R^{\times}$
of $X_{h_H}^{\eta}$ such that $\gamma\left(0\right)={\bf p_{S}}\left(x_{0}\right)$.
Now, we need to find the curves $\varphi\left(t\right)$ and $g\left(t\right)$
of steps $2$ and $3$. Consider the flat connection $A$ given by
\eqref{ad} and \eqref{ka}. Define
\[
\varphi\left(t\right):=\left(F,{\bf p_{S}}\right)^{-1}\left(s_{0},\gamma\left(t\right)\right),\quad\forall t\in\left(-\epsilon,\epsilon\right),
\]
with $s_{0}=F\left(x_{0}\right)$. Then, ${\bf p_{S}}\left(\varphi\left(t\right)\right)=\gamma\left(t\right)$
and $F\left(\varphi\left(t\right)\right)=s_{0}$. In particular, $\varphi\left(t\right)$
belongs to a level set of $F$ (of value $s_{0}\in\R^{\times}$),
and consequently its tangent vector belongs to $\left\langle dF\right\rangle ^{o}=\ker A$,
i.e.
\[
A\left(\varphi'\left(t\right)\right)=0\quad\forall t\in\left(-\epsilon,\epsilon\right).
\]
Then, the two parts of \eqref{2-2} are satisfied. Furthermore, since
$\gamma\left(0\right)={\bf p_{S}}\left(x_{0}\right)$, 
\[
\varphi\left(0\right)=\left(F,{\bf p_{S}}\right)^{-1}\left(s_{0},\gamma\left(0\right)\right)=\left(F,{\bf p_{S}}\right)^{-1}\left(F\left(x_{0}\right),{\bf p_{S}}\left(x_{0}\right)\right)=x_{0}.
\]
Thus, the step $2$ is complete.

In order to find the curve $g\left(t\right)$,
let us calculate $A\left(X_{H}^{\omega}\left(\varphi\left(t\right)\right)\right)$.
Using the decomposition $TS=\left\langle \Delta\right\rangle \oplus\left\langle dF\right\rangle ^{o}$, we have that 
\[
X_{H}^{\omega}=f\,\Delta+Z,
\]
with $f\in C^\infty(S)$ and $Z$ a vector field on $S$ such that $Z(F)=0.$
It follows that
\[
\left\{ F,H\right\} _{S}=X_{H}^{\omega}(F)=f\,\Delta(F)=f\,F,
\]
and consequently
\[
f=\frac{\{ F,H\} _{S}}{F}.
\]
Then
\[
A\circ X_{H}^{\omega}=f\,\Delta=\frac{\left\{ F,H\right\} _{S}}{F}\,\Delta.
\]
Writing $g\left(t\right)=\exp\left(\alpha\left(t\right)\right)$,
Eq.~\eqref{3-2} translates to
\[
\alpha'\left(t\right)=\frac{\left\{ H,F\right\} _{S}\left(\varphi\left(t\right)\right)}{s_{0}},\quad\alpha\left(0\right)=0,
\]
which has the solution
\[
\alpha\left(t\right)=\frac{1}{s_{0}}\int_{0}^{t}\left\{ H,F\right\} _{S}\left(\varphi\left(s\right)\right)\;ds.
\]
Summing up, the trajectory which we are looking for is
\begin{equation}
\Gamma\left(t\right)=\phi^{S}\left(\exp\left(\frac{1}{s_{0}}\int_{0}^{t}\left\{ H,F\right\} _{S}\left(\varphi\left(s\right)\right)\;ds\right),\varphi\left(t\right)\right),\quad\textrm{with}\quad \varphi\left(t\right)=\left(F,{\bf p_{S}}\right)^{-1}\left(s_{0},\gamma\left(t\right)\right).\label{exp}
\end{equation}

\begin{remark}
\label{r2} If $H$ itself is a scaling function (i.e.~$H\left(x\right)\neq 0$
for all $x$), then we can take $F=H$. In such a case $\left\{ H,F\right\} _{S}=0$
and consequently
\begin{equation}\label{G1}
\Gamma\left(t\right)=\varphi\left(t\right)=\left(H,{\bf p_{S}}\right)^{-1}\left(s_{0},\gamma\left(t\right)\right),
\end{equation}
where $s_{0}=H\left(x_{0}\right)$.
\end{remark}
Now, we shall construct an alternative expression of
$\Gamma\left(t\right)$, which involves the Reeb vector field $\mathcal{R}$
of $\left(C,\eta\right)$. Using \eqref{pxxp} and acting with the
first and last members on the differential of the contact Hamiltonian
function $h_H$ (related to $H$), it easily follows that
\[
\left\{ H,F\right\} _{S}=-F\,\left(\mathcal{R}\left(h_H\right)\circ {\bf p_{S}}\right).
\]
Then, 
\[
\frac{1}{s_{0}}\int_{0}^{t}\left\{ H,F\right\} _{S}\left(\varphi\left(s\right)\right)\;ds=-\int_{0}^{t}\mathcal{R}\left(h_H\right)\left(\gamma\left(s\right)\right)\;ds,
\]
and applying $\left(F,{\bf p_{S}}\right)$ on \eqref{exp} we have that
\begin{equation}
\left(F,{\bf p_{S}}\right)\left(\Gamma\left(t\right)\right)=\left(s_{0}\,\exp\left(-\int_{0}^{t}\mathcal{R}\left(h_H\right)\left(\gamma\left(s\right)\right)\;ds\right),\gamma\left(t\right)\right).\label{exp2}
\end{equation}
Thus, we have found, up to quadratures, the trajectories $\Gamma\left(t\right)$
of $X_{H}^{\omega}$ from the trajectories $\gamma\left(t\right)$
of $X_{h_H}^{\eta}$.
\begin{remark}

According to the local existence of scaling functions, if there is not a (global) scaling
function for $\phi^{S}$, then we can proceed as above around every
point $x\in S$, just replacing $S$ by an appropriate coordinate
neighborhood $U$ of $x_0$. In particular, we can obtain the result
of Remark \ref{r2} along the open submanifold of $S$ where $H\neq0$.
\end{remark}

To end this section, suppose that, instead of a symplectic Hamiltonian system, we have
a Poisson Hamiltonian system $(P,\Pi,H)$ with scaling symmetry $\phi^{P}:\R^{\times}\times P\to P$
such that ${\bf p_P}:P\rightarrow K=P/\R^{\times}$ is a principal
bundle. Assume that $F:P\rightarrow\R^{\times}$ is a scaling function
for $\phi^{P}$. Then, as we saw above, the related line bundle $\pi_{L}:L\to K$
is trivial (\textit{via} a global section as that given by \eqref{sigma}),
so the sections of $\pi_{L^{*}}$ can be identified with the functions
$h:K\rightarrow\mathbb{R}$, which in turn are in bijection with the
homogeneous functions $H:P\rightarrow\mathbb{R}$ through the equation
$h_{H}\circ{\bf p_P}=\frac{1}{F}H$. Also, the related Kirillov
bracket $\left[\cdot,\cdot\right]_{L^{*}}$ can be identified with
the Jacobi bracket $\left\{ \cdot,\cdot\right\} _{K}$ given by
\[
\left\{ h_{H_{1}},h_{H_{2}}\right\} _{K}\circ{\bf p_P}=-\frac{1}{F}\left\{ H_{1},H_{2}\right\} _{P}.
\]

Moreover, following the same calculations made along this section
for the symplectic case, given $x_{0}\in P$, we can construct the
trajectory $\Gamma\left(t\right)$ of $X_{H}^{\left\{ \cdot,\cdot\right\} _{P}}$
such that $\Gamma\left(0\right)=x_{0}$, in terms of the trajectory
$\gamma\left(t\right)$ of $X_{h_{H}}^{\left[\cdot,\cdot\right]_{L^{*}}}$such
that $\gamma\left(0\right)={\bf p_P}\left(x_{0}\right)$, through
the equation
\[
\left(F,{\bf p_P}\right)\left(\Gamma\left(t\right)\right)=\left(s_{0}\,\exp\left(-\int_{0}^{t}E\left(h_{H}\right)\left(\gamma\left(s\right)\right)\;ds\right),\gamma\left(t\right)\right),
\]
with $s_{0}=F\left(x_{0}\right)$ and $E\in\mathfrak{X}\left(K\right)$
such that $E\left(f\right)=\left\{ 1,f\right\} _{K}$.

\begin{example}{\bf The 2d harmonic oscillator. } {\rm If we consider the  local coordinates $(\rho,\theta,\rho',\theta')$ defined at the end 
of Example \ref{ex:STQho} on $\R^+ \times S^1\times \R^+ \times S^1\cong T^*(\R^2-\{(0,0)\})-0_{\R^2-\{(0,0)\}}$, then  we have that the local expression of the Hamiltonian function is just 
$$H(\rho,\theta,\rho',\theta')=\frac{1}{2}\rho^2(1 + (\rho')^2)$$
which is a scaling function. 

Moreover, the reduced space  is $\R^+  \times S^1\times  S^1$ and the principal bundle ${\bf p}:\R^+ \times S^1\times \R^+ \times S^1\to \R^+ \times S^1 \times S^1$ is given by ${\bf p}(\rho,\theta,\rho',\theta')=(\rho',\theta,\theta')$. Now, we will describe the integral curve $\Gamma:(-\epsilon,\epsilon)\to \R^+ \times S^1\times \R^+ \times S^1$ of 
$X_H^\omega$ with $\Gamma(0)=(\rho_0,\theta_0,\rho'_0,\theta'_0).$
Note that the inverse of the  diffeomophism 
$$(H,{\bf p}):\R^+ \times S^1\times \R^+\times S^1\to \R^+\times \R^+ \times S^1 \times S^1,\;\;\; 
(H,{\bf p})(\rho,\theta,\rho',\theta')=(\frac{1}{2}\rho^2(1 + (\rho')^2), \rho',\theta,\theta')$$
is 
$$(H,{\bf p})^{-1}:\R^+ \times \R^+\times S^1 \times S^1\to \R^+ \times S^1\times \R^+ \times S^1,\;\;\; 
(H,{\bf p})^{-1}(\rho,\rho',\theta,\theta')=(\sqrt\frac{2\rho}{1 + (\rho')^2}, \theta,\rho',\theta').$$

Therefore, the integral curve of   $X_H^\omega$ such that $\Gamma(0)=(\rho_0,\theta_0,\rho'_0,\theta'_0)$ is (see (\ref{G1}))

$$\Gamma(t)=\left(H,{\bf p}\right)^{-1}\left(\frac{1}{2}\rho_0^2(1 + (\rho_0')^2),\gamma\left(t\right)\right)=\left(\rho_0\sqrt
\frac{(1 + (\rho_0')^2)}{(1 + \rho'(t)^2)},\gamma(t)\right), $$
where  $\gamma(t)=(\theta(t),\rho'(t),\theta'(t))$ is the integral curve of the contact Hamiltonian vector field 
$$X_{h_H}^\eta=(1+(\rho')^2)\cos(\theta-\theta')\partial_{\rho'}+\sin(\theta-\theta')(\frac{1}{\rho'}\partial_{\theta'}+\rho'\partial_{\theta})$$
(see (\ref{p1})). 
 \hfill$\blacklozenge$}
\end{example}

\begin{example}{\bf The projective cotangent Hamiltonian system deduced from a standard linear Hamiltonian system. }
{\rm We consider Example \ref{ex:projective0} with $Y\in {\mathfrak X}(Q)$ a vector field on the manifold $Q$ of dimension $n.$ Let $U_{i_0}$ be the open subset of $T^*Q-0_Q$ 
given by
$$U_{i_0}=\{(q^1,\dots q^n,p_1,\dots p_{n})\in T^*Q-0_Q/p_{i_0}\not=0\},$$ 
with $(q^i,p_i)$ local coordinates on $T^*Q.$ The local expressions of the linear function $Y^\ell$ and of the corresponding Hamiltonian vector field $X^{\omega_Q}_{Y^\ell}$  are
$$Y^\ell(q,p)=Y^i(q)p_i \qquad\mbox{ and }\qquad X^{\omega_Q}_{Y^\ell}=Y^k\partial_{q^k} - p_j{\partial_{q^k}}Y^j\partial_{p_k},$$
 with  $Y(q)=Y^i(q){\partial_{ q^i}}.$ 
  
  Note that $Y^\ell$ is a scaling function if and only if $Y$ is a vector field without zeros. In any case, we have  a scaling function on $U_{i_0}$ given by
 $$F:U_{i_0}\to \R,\quad F(q^i,p_i)=p_{i_0}$$ 
   
  After the reduction process of the Hamiltonian symplectic  system $(T^*Q-0_Q,\omega_Q,H)$ by the scaling symmetry~(\ref{action1}), 
  we have that the local expressions of the reduced elements are:
\begin{itemize}
\item The local expression of the projective bundle  ${\bf p}:T^*Q-0_Q\to {\Bbb P}(T^*Q)$ on $U_{i_0}:$
$${\bf p}(q^1,\dots q^n, p_1,\dots p_n)=(q^1,\dots q^n,\frac{p_1}{p_{i_0}}, \dots,\frac{p_{i_0-1}}{p_{i_0}}, \frac{p_{i_0+1}}{p_{i_0}},\dots, \frac{p_n}{p_{i_0}}).$$

\item The contact distribution ${\mathcal D}$  on $U_{i_0}:$ 
$$
\begin{array}{rcl}
({\mathcal D}_{(q,\widetilde{p})})_{|{{\bf p}(U_{i_{0}})}}
=&&T_{(q,p)}{\bf p}(<p_idq^i>^o)=T_{(q,p)}{\bf p}<X_1,\dots,X_{i_0-1},X_{i_0+1},\dots, X_{n},
\partial_{p_1}, \dots, \partial_{p_n}>\\[5pt]
=&&<\widetilde{X}_1,\dots,\widetilde{X}_{i_0-1},\widetilde{X}_{i_0+1},\dots, \widetilde{X}_{n},
\partial_{\widetilde{p}_1},\dots, \partial_{ \widetilde{p}_{i_0-1}},\partial_{\widetilde{p}_{i_0+1}},\cdots,\partial_{\widetilde{p}_n}>,
\end{array}
$$
with $X_i=p_{i}\partial_{q^{i_{0}}}-p_{i_{0}}\partial_{q^{i}}$, 
$\widetilde{X}_i=\widetilde{p}_{i}\partial_{q^{i_{0}}}-{\widetilde{p}_{i_{0}}}\partial_{q^{i}}$ and 
$(q,\widetilde{p})=(q,\widetilde{p}_1, \cdots \widetilde{p}_{i_0-1}, \widetilde{p}_{i_0+1},\dots, \widetilde{p}_{n})$ local coordinates on ${\Bbb P}(T^*Q).$ 

The local expression of the line bundle $\pi_{{\mathcal D}^o}:{\mathcal D}^o\to  {\Bbb P}(T^*Q)$  on  $U_{i_0}$  is 
$$\pi_{{\mathcal D}^o}(q,\widetilde{p_i},t)=(q,\widetilde{p_i}).$$

\item The section $h_{Y^\ell}:{\Bbb P}(T^*Q)\to ({\mathcal D}^o)^* $ of $\pi_{({\mathcal D}^o)^*}:({\mathcal D}^o)^*\to  {\Bbb P}(T^*Q)$ associated with $Y^\ell:$  
\begin{equation}\label{ellp}
h_{Y^\ell}(q,\widetilde{p})(q,\widetilde{p},t)=Y^\ell(q,\widetilde{p}_1,\cdots, \widetilde{p}_{i_0-1}, t, \widetilde{p}_{i_0+1},\cdots, \widetilde{p}_{n})=Y^i(q)\widetilde{p}_i + Y^{i_0}(q)t.
\end{equation}

\item The ${\bf p}$-projection of the Hamiltonian vector field  $X^{\omega_Q}_{Y^\ell}\in {\mathfrak X}(T^*Q-0_Q)$   to ${\Bbb P}(T^*Q):$
\begin{equation}\label{Y}
Y^i\partial_{{q}^i}+ \big(\widetilde{p}_j(\widetilde{p}_i\partial_{{q}^{i_0}}Y^j - \partial_{{q}^i}Y^j)+ \widetilde{p}_i\partial_{q^{i_0}}Y^{i_0} -\partial_{{q}^i}Y^{i_0}
\big)\partial_{\widetilde{p}_i}.\end{equation}

\item The trivialization $(F,{\bf p}):T^*Q-0_Q\to \R^\times\times {\Bbb P}(T^*Q):$ 
$$(F,{\bf p}) (q^1,\dots ,q^n,p_1,\dots ,p_n)=(p_{i_0}, (q^1,\dots, q^n,\frac{p_1}{p_{i_0}}, \dots,\frac{p_{i_0-1}}{p_{i_0}}, \frac{p_{i_0+1}}{p_{i_0}},\dots, \frac{p_n}{p_{i_0}}))$$

and  its inverse map 
$$(F,{\bf p})^{-1} (s,({q}^1,\dots ,{q}^n,\widetilde{p}_1,\dots, \widetilde{p}_{i_0-1},\widetilde{p}_{i_0+1}, \dots, \widetilde{p}_n))= ({q}^1,\dots ,{q}^n,s\widetilde{p}_1,\dots, s\widetilde{p}_{i_0-1},s,s\widetilde{p}_{i_0+1}, \dots, s\widetilde{p}_n).$$
\end{itemize}
The integral curve $\Gamma:(-\epsilon,\epsilon)\to T^*Q-0_Q$ of the Hamiltonian vector field $X^{\omega_Q}_{Y^\ell}$ 
such that $\Gamma(0)=(q_0^i,p^0_i)$ is (see~(\ref{exp}))
$$
\Gamma\left(t\right)=(q^i(t), \exp\left(\frac{1}{p^0_{i_0}}\int_{0}^{t}(p_j\partial_{q^{i_0}} Y^j(q(s))ds)\right)({p}^0_{i_0}\widetilde{p}_1(t), \dots, {p}^0_{i_0}\widetilde{p}_{i_0-1}(t),{p}^0_{i_0},{p}^0_{i_0}\widetilde{p}_{i_0+1}(t),\dots, {p}^0_{i_0}\widetilde{p}_n(t)),
$$
where $\gamma(t)=(q^i(t),\widetilde{p}_1(t), \dots, \widetilde{p}_{i_0-1}(t),\widetilde{p}_{i_0+1}(t),\dots, \widetilde{p}_n(t))$ is an integral curve of the vector field given in (\ref{Y}) such that $$\gamma(0)=(q_0^i, \frac{p_1^0}{p_{i_0}^0},\dots,  \frac{p_{i_0}^0}{p_{i_0}^0},  \frac{p_{i_0+1}^0}{p_{i_0}^0}, \dots  ,\frac{p_n^0}{p_{i_0}^0}).$$

 \hfill$\blacklozenge$

 }
\end{example}

\section*{\textbf{Acknowledgments}}

\noindent The work of A.~Bravetti was partially supported by DGAPA-UNAM, program PAPIIT, Grant No.~IA-102823. 
S.~Grillo thanks SIIP-UNCuyo and CONICET, Argentina, for financial support. JC Marrero  and E.~Padr\'on acknowledge financial support from the Spanish Ministry of Science and Innovation under grant PID2022-137909NB-C22 and RED2022-134301-T (MCIN/AEI /10.13039/501100011033). Thanks  to  Maldacena program of visiting professors (JC Marrero) and  to the University of La Laguna (E.~Padr\'on)  for the financial support for a stay at the Balseiro Institute where this paper was finished.

\appendix

\section{Line bundles and $\R^\times$-principal bundles}\label{A}
{\rm 
Let  $p_M:M\to K$ be the principal bundle  associated with an action $\phi^M:\R^\times\times M\to M$  of the multiplicative group 
$\R^\times$ (with $\R^\times=\R-\{0\}$ or $\R^\times=\R^+$) on the manifold $M.$  Consider the representation $\R^\times\times \R\to \R$  of $\R^\times$ 
over the vector space of real numbers  given by $$(s,t)\to \frac{t}{s}.$$ 


Let $\tilde\phi^M:\R^\times \times(M\times \R)\to (M\times \R)$ be the action  of $\R^\times$ on the cartesian product $M\times \R$ given by 
\begin{equation}\label{acvec}
\tilde{\phi}^M(s,(x,t))=(\phi^M(s,x),\frac{t}{s})\mbox{ with }(s,(x,t))\in \R^\times \times (M\times \R).
\end{equation}

Then, the first projection $p_1:M\times \R\to M$ is an equivariant map with respect to the actions $\tilde{\phi}^M$ and $\phi^M$ and the map 
$\pi_L:L:=(M\times \R)/\R^\times\to K=M/\R^\times$ between the corresponding quotient spaces is a vector bundle with fiber $\R.$ 
It is {\it the  line bundle associated with $p_M:M\to K$ and the representation (\ref{acvec})}. 

If $0_L$ is the zero section of the vector bundle $\pi_L: L\to K$ and $\pi:M\times \R\to L=(M\times \R)/\R^\times $ is the quotient map,  
one can identify $M$ with $L-0_L$,  via the isomorphism of principal bundles
$$M\to (L-0_L),\qquad  x\in M \to \pi(x,1)\in L-0_L.$$

Conversely, if $\pi_L:L\to K$ is a line bundle (vector bundle with fiber $\R$) and $0_L$ is the zero section of $\pi_L$, then $p_M:M:=(L-0_L)\to K$ is a $\R^\times$-principal bundle.
 The action associated with this principal bundle is given by 
$$\phi^{M}:\R^\times \times (L-0_L)\to (L-0_L),\qquad \phi^M(s,x)=sx\,,$$
for $x\in L-0_L$,  and the line bundle associated with this principal bundle is isomorphic to $\pi_L.$ 
In fact, the $\R^\times$-invariant map 
$$ (L-0_L)\times \R\to L,\;\;\; (x,t)\to tx, \mbox{ with  } (x,t)\in  (L-0_L)\times \R,$$
induces an isomorphism between the line bundles $((L-0_L)\times \R)/\R^\times$ and $L.$ 

\begin{proposition}\label{1-1}\rm
Let $p_M:M\to K$ be a $\R^\times$-principal  bundle and $\pi_L:L\to K$ its associated  line bundle. Then, there is a one-to-one correspondence 
between the  sections $h:K\to L^*$  on the dual vector bundle of $\pi_L:L\to K$ and the homogeneous functions  on $M$, 
i.e.~functions $H:M\to \R$ satisfying  the condition 
$$H(\phi^M (s,x))=sH(x),\;\;\; \mbox{ for all } s\in \R^\times, x\in P,$$
where $\phi^M:\R^\times \times M\to M$ is the corresponding principal action. 
  \end{proposition}
\begin{proof}

Indeed, if $h:K\to L^*$ is a section of $\pi_{L^*}:L^*\to K$ and $\pi:M\times \R\to L={(M\times \R)}/{\R^\times}$ is the canonical projection,  one can define the function 
\begin{equation}\label{Hh}
H_h: M\to \R,\;\;\; H_h(x)=h(p_M(x))(\pi(x,1)),\mbox{ for all }x\in M,
\end{equation}
which satisfies that
$$H_h(\phi^M(s,x))=h(p_M(x))(\pi(\phi^M(s,x),1))=h(p_M(x))(\pi(x,s))=h(p_M(x))(s\pi(x,1))=sH_h(x)$$
for $(s,x)\in \R^\times\times M.$ Therefore, $H_h$ is homogenous with respect to $\phi^M.$ 

Conversely, if $H:M\to \R$ is a homogenous function for the action $\phi^M$,  then we have  a section $h_H:K\to L^*$ of $\pi_{L^*}$ given by 
\begin{equation}\label{hH}
h_H(p_M(x))(\pi(x,t))=tH(x)\mbox{ for all }x\in M \mbox{ and } t\in \R,
\end{equation}
which is well-defined by the homogeneity of $H.$  
\end{proof}

%
%
%
%


\end{document}